\documentclass[10pt]{amsart}

\usepackage{bbm}
\usepackage{eucal}
\numberwithin{equation}{section}
\usepackage{xcolor}
\usepackage[colorlinks=true]{hyperref}
\hypersetup{linkcolor=black, citecolor=black}
\usepackage{amsmath}
\usepackage{amsthm}
\usepackage{amssymb}
\usepackage{mathrsfs}
\usepackage{tikz-cd}
\newtheorem{thm}{Thm}[section]

\newtheorem{theorem}[thm]{Theorem}
\newtheorem{proposition}[thm]{Proposition}
\newtheorem{corollary}[thm]{Corollary}

\newtheorem{lemma}[thm]{Lemma}
\theoremstyle{definition}

\newtheorem{definition}[thm]{Definition}
\newtheorem{remark}[thm]{Remark}
\newtheorem{example}[thm]{Example}
\newtheorem{construction}[thm]{Construction}

\theoremstyle{remark}

\newcommand{\A}{\mathbb{A}}

\newcommand{\G}{\mathbb{G}}

\newcommand{\N}{\mathbb{N}}
\renewcommand{\P}{\mathbb{P}}

\newcommand{\T}{\mathbb{T}}

\newcommand{\Z}{\mathbb{Z}}

\newcommand{\cC}{\mathcal{C}}

\newcommand{\cF}{\mathcal{F}}
\newcommand{\cG}{\mathcal{G}}

\newcommand{\cI}{\mathcal{I}}

\newcommand{\cM}{\mathcal{M}}

\newcommand{\cO}{\mathcal{O}}
\newcommand{\cP}{\mathcal{P}}
\newcommand{\cQ}{\mathcal{Q}}
\newcommand{\cR}{\mathcal{R}}
\newcommand{\cS}{\mathcal{S}}
\newcommand{\cT}{\mathcal{T}}
\newcommand{\cU}{\mathcal{U}}
\newcommand{\cV}{\mathcal{V}}
\newcommand{\cW}{\mathcal{W}}
\newcommand{\cX}{\mathcal{X}}
\newcommand{\cY}{\mathcal{Y}}
\newcommand{\cZ}{\mathcal{Z}}

\newcommand{\et}{{\acute{e}t}}

\DeclareMathOperator{\Hom}{Hom}

\DeclareMathOperator{\Spec}{Spec}
\DeclareMathOperator{\coker}{coker}

\newcommand{\id}{{\rm id}}

\newcommand{\lSch}{\mathbf{lSch}}
\newcommand{\lSpc}{\mathbf{lSpc}}
\newcommand{\lFan}{\mathbf{lFan}}
\newcommand{\lSmSpc}{\mathbf{lSmSpc}}

\newcommand{\lSmFan}{\mathbf{lSmFan}}
\newcommand{\rank}{\mathrm{rank}}

\newcommand{\sSm}{\mathbf{sSm}}

\newcommand{\ol}{\overline}
\newcommand{\pt}{\mathrm{pt}}
\newcommand{\Blow}{\mathrm{Bl}}

\newcommand{\lSm}{\mathbf{lSm}}
\newcommand{\Coeq}{\mathrm{Coeq}}
\newcommand{\Eq}{\mathrm{Eq}}
\newcommand{\gp}{\mathrm{gp}}

\newcommand{\Deform}{\mathrm{D}}
\newcommand{\Normal}{\mathrm{N}}

\newcommand{\colimit}{\mathop{\mathrm{colim}}}

\newcommand{\boxx}{\square}

\newcommand{\Shv}{\mathbf{Shv}}

\newcommand{\ul}{\underline}

\begin{document}
\title{Inverting log blow-ups in log geometry}
\author{Doosung Park}
\address{Department of Mathematics and Informatics, University of Wuppertal, Germany}
\email{dpark@uni-wuppertal.de}
\date{\today}
\begin{abstract}
In the category of log schemes, it is unclear how to define the blow-ups for non-strict closed immersions.
In this article, we introduce the notion of divided log spaces.
We obtain the category of divided log spaces by locally inverting log blow-ups in the category of log schemes.
We show that blow-ups exist for closed immersions of log smooth divided log spaces.
This is an ingredient of the motivic six-functor formalism for log schemes.
\end{abstract}
\maketitle
\section{Introduction}

For a regular embedding $Z\to X$ of schemes, the deformation to the normal cone construction, denoted $\Deform_Z X$, plays a central role in intersection theory.
For example, if $X$ is a smooth scheme over a field $k$, Fulton-MacPherson \cite{zbMATH01027930} used this construction for the diagonal morphism $X\to X\times_k X$ to define the intersection product.

The notion of log schemes, introduced by Fontaine-Illusie and further developed by Kato \cite{zbMATH00125808}, can be thought of as the notion of ``schemes with boundaries''.
The extra structure of boundaries is helpful for the compactification and degeneration problems in algebraic geometry.
For example, log geometry has been applied to the construction of the fine moduli space of polarized logarithmic Hodge structures with extra structures \cite{zbMATH05485422} and the proof of the $C_{st}$-conjecture in $p$-adic hodge theory \cite{MR1705837}.

To develop intersection theory or motivic homotopy theory of log schemes, the deformation to the normal cone construction for log schemes is desirable.
However, there is a technical difficulty:
If $X$ is a log smooth fs log scheme over a field $k$ whose log structure is nontrivial, then the diagonal morphism $X\to X\times_k X$ is a closed immersion that is not strict.
It is unclear how to define the blow-ups for non-strict closed immersions in the category of fs log schemes because a non-strict closed immersion is not defined by a sheaf of ideals.
Hence the construction of $\Deform_X(X\times_k X)$ in this category is unclear.

The purpose of this article is to introduce the notion of \emph{divided log spaces}.
The category of divided log spaces $\lSpc$ is obtained by locally inverting log blow-ups in the category of fs log schemes $\lSch$.
More precisely, we consider a full subcategory $\lFan$ of $\lSch$ in Definition \ref{lFan.1} such that globally inverting log blow-ups in $\lFan$ is reasonable.
Then we define a divided space in Definition
\ref{equiv.5} as a presheaf on $\lFan$ satisfying certain conditions whose formulation is similar to that of algebraic spaces.
Any morphism of divided log spaces behaves like an exact morphism of fs log schemes.
In particular, a non-strict closed immersion of fs log schemes becomes a strict closed immersion of divided log spaces.

For applications, we construct the open complements (resp.\ blow-ups) for non-strict closed immersions $Z\to X$ of fs log schemes (resp.\ log smooth fs log schemes) in the category of divided log spaces.
This allows us to construct $\Deform_Z X$, which is a crucial ingredient in  \cite{logGysin} and \cite{logsix}.

\subsection*{Organization of the article.}

In Section \ref{dZar}, we explain the three types of covers in log geometry: dividing covers, Zariski covers, and dividing Zariski covers.
The associated topologies are called the dividing, Zariski, and dividing Zariski topologies.
We study several properties of dividing Zariski sheaves.

In Section \ref{rep}, we consider properties  of representable morphisms of sheaves.
For example, we define representable strict closed immersions, representable log smooth morphisms, and representable Zariski covers of sheaves.
We also provide various basic lemmas about representable morphisms, which are used in later sections.

The definition of divided log spaces appears in Section \ref{divspace}.
A sheaf $\cX$ is called a divided log space if the diagonal morphism is a representable strict closed immersion, and if there exists a representable Zariski cover $\cY\to \cX$ such that $\cY$ is representable.
Properties of representable morphisms of sheaves can be restricted to divided log spaces.

The purpose of Section \ref{equiv} is to glue divided log spaces.
Our method for this is to introduce the notion of Zariski equivalence relations, which is an analog of \'etale equivalence relations in the theory of algebraic spaces.

In Section \ref{property}, we explain properties of morphisms of divided log spaces, which do not need to be representable.
For example, we define closed immersions, log smooth morphisms, and Zariski covers of divided log spaces.
We show that closed immersions of divided log spaces are representable strict closed immersions.

In Section \ref{topology}, we introduce several topologies on the category of divided log spaces.
We also compare sheaves on the category of divided log spaces and sheaves on the category of fs log schemes.

In Sections \ref{complement} and \ref{blow-up}, we define the open complements and blow-ups of closed immersions of divided log spaces using universal properties.
The open complements always exist, but we only show the existence of the blow-ups in the case of closed immersions of log smooth divided log spaces.
We combine these two notions for the deformation to the normal cone construction.
We also show that the open complements are closed under pullbacks and the blow-ups are closed under pullbacks along log smooth morphisms.

In Appendices, we collect several results in log geometry.

\subsection*{Related work}
Kato introduced \emph{algebraic valuative log spaces} in \cite{Katoval}, and this also does a similar job of locally inverting log blow-ups.
One technical advantage of our divided log spaces is that the definition resembles that of algebraic spaces.
Hence we can imitate many proofs in the literature of algebraic spaces to develop our theory.

Kato gave in \cite[Proposition 1.4.2]{Katoval} a description of the Hom group in the category of algebraic valuative log spaces, whose proof was left to the reader.
Assuming this, we expect that there is a fully faithful functor from the category of divided log spaces into the category of algebraic valuative log spaces since we have a similar result in Proposition \ref{div.5}.
We left an investigation about the comparison to the interested reader.

\subsection*{Notation and conventions}

Throughout this article, we fix a noetherian fs log scheme $B$ of finite Krull dimension.
Our standard reference for the notation and terminology in log geometry is Ogus's book \cite{Ogu}.
We also follow the convention in \cite[Definition 1.9]{MR1457738} for several properties of fs log schemes and morphisms of fs log schemes except open immersions.
A morphism of fs log schemes $f\colon Y\to X$ is an \emph{open immersion} if it is strict and its underlying morphism of schemes $\ul{f}\colon \ul{Y}\to \ul{X}$ is an open immersion.

Every fs log scheme in this article is equipped with a Zariski log structure except in Remark \ref{blow.8}, Proposition \ref{blow.11}, and Lemma \ref{blow.9}.

The coproduct $M\oplus_P N$ of saturated monoids is taken in the category of saturated monoids, and the fiber product $X\times_S Y$ of saturated log schemes is taken in the category of saturated log schemes.

\subsection*{Acknowledgements} The author wishes to thank the anonymous referees for their very careful reading of the manuscript and many valuable comments.
This research was conducted in the framework of the DFG-funded research training group GRK 2240: \emph{Algebro-Geometric Methods in Algebra, Arithmetic and Topology}.

\section{Dividing Zariski topology}
\label{dZar}
We want a topology that is finer than the Zariski topology and does the job of locally inverting log blow-ups in the category of sheaves.
The dividing Zariski topology defined below is suited for this.

\begin{definition}
\label{mono.1}
Let $f\colon Y\to X$ be a quasi-compact morphism of fs log schemes.
\begin{enumerate}
\item[\textup{(1)}] We say that $f$ is a \emph{dividing cover} if $f$ is a universally surjective proper log \'etale monomorphism.
\item[\textup{(2)}] We say that $f$ is a \emph{Zariski cover} if $f$ is surjective and of the form $\amalg_{i\in I}Y_i\to X$ with finite $I$ such that each $Y_i\to X$ is an open immersion.
\item[\textup{(3)}] We say that $f$ is a \emph{dividing Zariski cover} if $f$ is universally surjective and of the form $\amalg_{i\in I}Y_i\to X$ with finite $I$ such that each $Y_i\to X$ is a log \'etale monomorphism.
\end{enumerate}
\end{definition}

Recall that a morphism $f\colon Y\to X$ in a category with fiber products is a monomorphism if and only if the diagonal morphism $Y\to Y\times_X Y$ is an isomorphism.

\begin{definition}
\label{mono.2}
Let $\{Y_i\to X\}_{i\in I}$ be a family of quasi-compact morphisms of fs log schemes with finite $I$.
\begin{enumerate}
\item[\textup{(1)}] The family is called a \emph{Zariski covering family} if $\amalg_{i\in I}Y_i\to X$ is a Zariski cover.
\item[\textup{(2)}] The family is called a \emph{dividing Zariski covering family} if $\amalg_{i\in I}Y_i\to X$ is a dividing Zariski cover.
\end{enumerate}
\end{definition}

Every dividing cover is a dividing Zariski cover.
Every pullback of a dividing (resp.\ dividing Zariski) cover is again a dividing (resp.\ dividing Zariski) cover.
Every composition of dividing (resp.\ dividing Zariski) covers is again a dividing (resp.\ dividing Zariski) cover.

\begin{definition}
The topology on the category of quasi-compact fs log schemes generated by dividing (resp.\ dividing Zariski) covers is called the \emph{dividing} (resp.\ \emph{dividing Zariski}) \emph{topology}.
Let $div$ (resp.\ $dZar$) be the shorthand for the dividing (resp.\ dividing Zariski) topology.
\end{definition}

\begin{definition}
We refer to \cite[Definition II.1.9.2]{Ogu} for the definition of the category of fans.
For a fan $\Sigma$, let $\T_{\Sigma}$ be the fs log scheme whose underlying scheme is the toric scheme over $\Spec(\Z)$ associated with $\Sigma$ and whose log structure is the compactifying log structure \cite[Definition III.1.6.1]{Ogu} associated with the open immersion from the torus $\G_m^n$, where $n$ is the rank of $\Sigma$.

Every morphism of fans $\theta\colon \Delta\to \Sigma$ induces a morphism of fs log schemes $\T_{\theta}\colon \T_\Delta\to \T_\Sigma$.
We say that $\theta$ is a \emph{subdivision} if the associated homomorphism of lattices and the associated map of supports $\vert\Delta\vert\to \vert\Sigma\vert$ are isomorphisms.
In this case, the underlying morphism of schemes $\ul{\T_\theta}$ is a proper birational morphism.
\end{definition}

\begin{definition}
\label{lFan.1}
For an fs log scheme $X$, a \emph{fan chart of $X$} is a fan $\Sigma$ together with a strict morphism $X\to \T_\Sigma$.

Let $\lSch/B$ be the category of noetherian fs log schemes of finite Krull dimensions over $B$.
Let $\lFan/B$ be the full subcategory of the category $\lSch/B$ consisting of disjoint unions $\amalg_{i\in I} X_i$ with finite $I$ such that each $X_i$ admits a fan chart.
The dividing topology and dividing Zariski topology can be restricted to $\lFan/B$.

Any fs log scheme admits a fan chart Zariski locally by \cite[Theorem III.1.2.7(1)]{Ogu}.
Hence any $X\in \lSch/B$ admits a Zariski cover $Y\to X$ with $Y\in \lFan/B$. 
\end{definition}

\begin{definition}
Let $f\colon Y\to X$ be a morphism of fs log schemes.
A \emph{fan chart of $f$} is a triple $(\Sigma,\Delta,\theta\colon \Delta\to \Sigma)$ such that the induced diagram
\[
\begin{tikzcd}
Y\ar[d]\ar[r,"f"]&
X\ar[d]
\\
\T_\Delta\ar[r,"\T_\theta"]&
\T_\Sigma
\end{tikzcd}
\]
commutes, the vertical morphisms are strict, and $\theta$ is a morphism of fans.
If $X$ has a fan chart, then $f$ has a fan chart Zariski locally on $Y$ by \cite[Theorem III.1.2.7(1)]{Ogu}.
\end{definition}

\begin{example}
Every log blow-up \cite[Definition III.2.6.2]{Ogu} is a dividing cover, see \cite[Example A.11.1]{logDM}.

Let $\theta \colon \Sigma'\to \Sigma$ be a subdivision of fans.
We have isomorphisms
\[
\T_{\Sigma'}\simeq \T_{\Sigma'\times_{\Sigma} \Sigma'}\simeq \T_{\Sigma'}\times_{\T_\Sigma}\T_{\Sigma'},
\]
so the induced morphism $\T_\theta\colon \T_{\Sigma'}\to \T_\Sigma$ is a monomorphism.
Zariski locally, $\T_\theta$ is of the form $\A_\eta\colon \A_Q\to \A_P$ for some injective homomorphism $\eta\colon P\to Q$ of fs monoids such that $\eta\colon P^\gp\to Q^\gp$ is an isomorphism.
Hence \cite[Corollary IV.3.1.10]{Ogu} shows that $\T_\theta$ is log \'etale.
Since $\theta$ is a subdivision, $\T_\theta$ is proper.
As a consequence of \cite[Lemma A.11.4]{logDM}, $\T_\theta$ is universally surjective.
Hence we have shown that $\T_\theta$ is a dividing cover.

If $X\to \T_\Sigma$ is any morphism of quasi-compact log schemes, then the projection $X\times_{\T_\Sigma}\T_{\Sigma'}\to X$ is a dividing cover too.
\end{example}

\begin{proposition}
\label{equiv.11}
We have the following list of synonyms of morphisms of fs log schemes:
\begin{enumerate}
\item[\textup{(1)}] exact proper monomorphism $=$ strict closed immersion,
\item[\textup{(2)}] exact log \'etale monomorphism $=$ open immersion,
\item[\textup{(3)}] exact dividing cover $=$ isomorphism,
\item[\textup{(4)}] exact dividing Zariski cover $=$ Zariski cover.
\end{enumerate}
\end{proposition}
\begin{proof}
(1) Consequence of Proposition \ref{equiv.32} and \cite[Corollaire IV.18.12.6]{EGA}.

(2) Consequence of Proposition \ref{equiv.32} and \cite[Th\'eor\`eme IV.17.9.1]{EGA}.

(3) Consequence of (1) and (2).

(4) Consequence of (2).
\end{proof}

The case of log \'etale monomorphisms in the result below is \cite[Proposition A.11.5]{logDM}.

\begin{proposition}
\label{equiv.23}
Let $f\colon Y\to X$ be a quasi-compact morphism of fs log schemes.
Assume that $X$ admits a fan chart $\Sigma$.
If $f$ is a monomorphism (resp.\ proper monomorphism, resp.\ log \'etale monomorphism, resp.\ dividing cover, resp.\ dividing Zariski cover), then there exists a subdivision $\Sigma'$ of $\Sigma$ such that the pullback
\[
f'\colon Y\times_{\T_\Sigma}\T_{\Sigma'}\to X\times_{\T_\Sigma}\T_{\Sigma'}
\]
is a strict monomorphism (resp.\ strict closed immersion, resp.\ open immersion, resp.\ isomorphism, resp.\ Zariski cover).
\end{proposition}
\begin{proof}
Combine \cite[Proposition 4.2.3]{logA1} (a variant of \cite[Theorem III.2.6.7]{Ogu}) and Propositions \ref{equiv.32} and \ref{equiv.11}.
\end{proof}

This immediately implies the following.

\begin{corollary}
\label{equiv.34}
Let $f\colon Y\to X$ be a dividing cover of quasi-compact fs log schemes.
If $X$ admits a fan chart $\Sigma$, then $f$ admits a refinement of the form $X\times_{\T_\Sigma}\T_{\Sigma'}\to X$ for some subdivision $\Sigma'$ of $\Sigma$.
\end{corollary}

\begin{remark}
\label{equiv.35}
Let $\theta\colon \Sigma'\to \Sigma$ be a subdivision of fans.
Combine toric resolution of singularities \cite[Theorem 11.1.9]{CLStoric} and De Concini-Procesi's Theorem \cite[pp.\ 39--40]{TOda} to see that there exists a refinement $\Sigma''\to \Sigma$ of $\theta$ that is a composition of star subdivisions.
The induced morphism $\T_{\Sigma''}\to \T_{\Sigma}$ is a log blow-up.
Hence Proposition \ref{equiv.23} is valid if we further require that $\T_{\Sigma'}\to \T_{\Sigma}$ is a log blow-up.

Using this, one can check that the dividing Zariski topology on the category of quasi-compact fs log schemes is equivalent to the \emph{Zariski valuative topology} \cite[Definition 2.23]{zbMATH05342987}, which is the smallest Grothendieck topology containing morphisms of the form $f\colon \amalg_{i\in I} Y_i\to X$ with finite $I$ such that $f$ is universally surjective and each $Y_i\to X$ is a composition of open immersions and log blow-ups.

Due to \cite[Remark 2.24]{zbMATH05342987}, the dividing Zariski topology has enough points.
\end{remark}

The next result justifies our choice of the terminology ``dividing Zariski.''

\begin{proposition}
\label{div.8}
Let $\cF$ be a presheaf on $\lFan/B$.
Then $\cF$ is a dividing Zariski sheaf if and only if $\cF$ is both a dividing sheaf and a Zariski sheaf.
\end{proposition}
\begin{proof}
The only if direction is trivial.
For the if direction, assume that $\cF$ is both a dividing sheaf and a Zariski sheaf.
Let $Y\to X$ be a dividing Zariski cover in $\lFan/B$.
By Proposition \ref{equiv.23}, there exists a dividing cover $X'\to X$ such that the projection $Y':=Y\times_X X'\to X'$ is a Zariski cover.
Since $\cF$ is a Zariski sheaf, we obtain
\[
\cF(X')
\xrightarrow{\simeq}
\Eq(\cF(Y')\rightrightarrows \cF(Y'\times_{X'}Y')).
\]
Use the assumption that $\cF$ is a dividing sheaf to obtain
\[
\cF(X)
\xrightarrow{\simeq}
\Eq(\cF(Y)\rightrightarrows \cF(Y\times_{X}Y)).
\]
This shows that $\cF$ is a dividing Zariski sheaf.
\end{proof}

The dividing sheafification and dividing Zariski sheafification admit explicit descriptions as follows. 

\begin{proposition}
\label{div.10}
Let $\cF$ be a presheaf on $\lFan/B$.
For every $X\in \lFan/B$ admitting a fan chart $\Sigma$, there exists an isomorphism
\begin{equation}
a_{div}\cF(X)
\simeq
\colimit_{\Sigma'\to \Sigma}\cF(X\times_{\T_\Sigma}\T_{\Sigma'}),
\end{equation}
where the colimit runs over the category of the subdivisions of $\Sigma$.
\end{proposition}
\begin{proof}
Let $L_{div}\cF$ be the separated presheaf associated with $\cF$ for the dividing topology, which is defined using \cite[Section II.3.0.5]{SGA4}.
For a subdivision $\Sigma'$ of $\Sigma$, we set $X_{\Sigma'}:=X\times_{\T_\Sigma}\T_{\Sigma'}$.
By Corollary \ref{equiv.34}, we have isomorphisms
\begin{equation}
\label{div.10.1}
L_{div}\cF(X)
\simeq
\colimit_{\Sigma'\to \Sigma}\Eq(\cF(X_{\Sigma'})\rightrightarrows \cF(X_{\Sigma'}\times_X X_{\Sigma'}))
\simeq
\colimit_{\Sigma'\to \Sigma}\cF(X_{\Sigma'}).
\end{equation}

Suppose that $Y\to X$ is a dividing cover in $\lFan/B$.
Use Proposition \ref{equiv.23} and the above description \eqref{div.10.1} to obtain an isomorphism 
\[
L_{div}\cF(X)
\simeq
L_{div}\cF(Y).
\]
This means that $L_{div}\cF$ is already a dividing sheaf, i.e., $L_{div}\cF\simeq a_{div}\cF$.
\end{proof}

For a fan $\Sigma$, the category of subdivisions of $\Sigma$ is a filtered category.
Hence the colimit in \eqref{div.10.1} is filtered.

\begin{proposition}
\label{div.3}
Let $\cF$ be a Zariski sheaf on $\lFan/B$.
For every $X\in \lFan/B$ admitting a fan chart $\Sigma$, there exists an isomorphism
\begin{equation}
\label{div.3.1}
a_{dZar}\cF(X)
\simeq
\colimit_{\Sigma'\to \Sigma}\cF(X\times_{\T_\Sigma}\T_{\Sigma'}),
\end{equation}
where the colimit runs over the category of the subdivisions of $\Sigma$.
\end{proposition}
\begin{proof}
Suppose that $X'\to X$ is a Zariski cover.
For any subdivision $\Sigma'$ of $\Sigma$, we set $X_{\Sigma'}:=X\times_{\T_{\Sigma}}\T_{\Sigma'}$ and $X'_{\Sigma'}:=X'\times_X X_{\Sigma'}$.
We obtain an isomorphism
\[
\cF(X_{\Sigma'})
\xrightarrow{\simeq}
\Eq(\cF(X_{\Sigma'}')\rightrightarrows \cF(X_{\Sigma'}'\times_{X_{\Sigma'}}X_{\Sigma'}')).
\]
Since any filtered colimits commute with finite limits, we obtain
\[
\colimit_{\Sigma'\to \Sigma}\cF(X_{\Sigma'})
\xrightarrow{\simeq}
\Eq(\colimit_{\Sigma'\to \Sigma}\cF(X_{\Sigma'}')\rightrightarrows \colimit_{\Sigma'\to \Sigma}\cF(X_{\Sigma'}'\times_{X_{\Sigma'}}X_{\Sigma'}')).
\]
This shows that the right-hand side of \eqref{div.3.1} is a Zariski sheaf.
Together with Proposition \ref{div.8}, we finish the proof.
\end{proof}

The following result enables us to work with $\lFan/B$ instead of $\lSch/B$ in many situations.

\begin{proposition}
\label{div.9}
Let $t$ be the topology on $\lSch/B$ finer than the Zariski topology.
Then the inclusion functor $\lFan/B \to \lSch/B$ induces an equivalence
\[
\Shv_t(\lFan/B)
\simeq
\Shv_t(\lSch/B).
\]
\end{proposition}
\begin{proof}
For every $X\in \lSch/B$, there exists a Zariski cover $Y\to X$ such that $Y\in \lFan/B$.
The implication (i)$\Rightarrow$(ii) in \cite[Th\'eor\`eme 4.1]{SGA4} finishes the proof.
\end{proof}

\begin{definition}
\label{div.4}
For $X\in \lSch/B$, let $h_X$ be the dividing Zariski sheafification of the presheaf on $\lFan/B$ represented by $X$.
\end{definition}

If $f\colon Y\to X$ is a dividing cover in $\lFan/B$, then the induced morphism $h_f\colon h_Y\to h_X$ is an isomorphism.
Hence the Yoneda functor
\begin{equation}
\label{div.0.1}
h\colon \lSch/B \to \Shv_{dZar}(\lSch/B)
\simeq
\Shv_{dZar}(\lFan/B)
\end{equation}
is \emph{not} conservative if $B$ is nonempty.
In particular, $h$ is not fully faithful, and the dividing Zariski topology is not subcanonical.

\begin{proposition}
\label{div.5}
Suppose $X\in \lSch/B$ and $Y\in \lFan/B$.
If $Y$ admits a fan chart $\Sigma$, then there exists a canonical isomorphism
\[
h_X(Y)
\simeq
\colimit_{\Sigma'\to \Sigma}\Hom_{\lSch/B}(Y\times_{\T_\Sigma}\T_{\Sigma'},X),
\]
where the colimit runs over the category of subdivisions of $\Sigma$.
\end{proposition}
\begin{proof}
Since the presheaf represented by $X$ is a Zariski sheaf, we can use Proposition \ref{div.3} for $h_X$.
\end{proof}

\begin{proposition}
\label{div.13}
The Yoneda functor $h\colon \lSch/B \to \Shv_{dZar}(\lFan/B)$ preserves finite limits.
\end{proposition}
\begin{proof}
Follows from Proposition \ref{div.5} since filtered colimits commute with finite limits.
\end{proof}

\section{Representable morphisms of sheaves}\label{rep}

The purpose of this section is to deal with properties of morphisms of sheaves, which is needed later to define divided log spaces.

\begin{definition}
We say that $\cF\in \Shv_{dZar}(\lFan/B)$ is \emph{represented by $X\in \lFan/B$} if $\cF\simeq h_X$.
We say that a morphism $\cG\to \cF$ in $\Shv_{dZar}(\lFan/B)$ is \emph{represented by a morphism $f\colon Y\to X$ in $\lFan/B$} if there exists a commutative square with vertical isomorphisms
\[
\begin{tikzcd}
h_Y\ar[d,"\simeq"']\ar[r,"h_f"]&
h_X\ar[d,"\simeq"]
\\
\cG\ar[r]&
\cF.
\end{tikzcd}
\]
\end{definition}

\begin{definition}
\label{equiv.6}
Let $u\colon \cG\to \cF$ be a morphism in $\Shv_{dZar}(\lFan/B)$.
We say that $u$ is \emph{$\lFan/B$-representable} (or simply \emph{representable}) if for every morphism $h_X\to \cF$ with $X\in \lFan/B$,
there exists $Y\in \lFan/B$ such that $h_X\times_{\cF}\cG\simeq h_Y$.

Suppose that $\cP$ is a class of morphisms in $\lFan/B$.
We say that $u$ is a \emph{representable $\cP$-morphism} if for every morphism $h_X\to \cF$ with $X\in \lFan/B$, there exists a commutative square
\begin{equation}
\label{equiv.6.1}
\begin{tikzcd}
h_{V}\ar[d,"\simeq"']\ar[r,"h_{g}"]&
h_{U}\ar[d,"\simeq"]
\\
h_X\times_{\cF}\cG\ar[r,"p"]&
h_X
\end{tikzcd}
\end{equation}
with vertical isomorphisms such that $p$ is the projection and $g$ is a $\cP$-morphism,
i.e., $g$ is a morphism in $\cP$.
\end{definition}

Observe that every representable $\cP$-morphism is representable.
If $\cP$ is closed under pullbacks, then the class of representable $\cP$-morphisms in $\Shv_{dZar}(\lFan/B)$ is closed under pullbacks too.

\begin{proposition}
\label{equiv.30}
Suppose that $\cP$ is the class of isomorphisms in $\lFan/B$.
If $f\colon \cG\to \cF$ is a representable $\cP$-morphism in $\Shv_{dZar}(\lFan/B)$, then $f$ is an isomorphism.
\end{proposition}
\begin{proof}
For every morphism $h_X\to \cF$ with $X\in \lFan/B$, we have a cartesian square
\[
\begin{tikzcd}
h_X\ar[r,"\id"]\ar[d]&
h_X\ar[d]
\\
\cG\ar[r,"f"]&
\cF.
\end{tikzcd}
\]
This means $\cG(X)\simeq \cF(X)$.
\end{proof}

The next five lemmas deal with the structure of representable morphisms.

\begin{lemma}
\label{equiv.24}
Let $f\colon h_Y\to h_X$ be a morphism in $\Shv_{dZar}(\lFan/B)$, where $X,Y\in \lFan/B$.
Then there exists a dividing cover $v\colon V\to Y$ and a morphism $g\colon V\to X$ in $\lFan/B$ such that $fh_v=h_g$.
Furthermore, $f$ is representable.
\end{lemma}
\begin{proof}
The first claim is a consequence of Proposition \ref{div.5}.
To show that $f$ is representable, using the first claim, we only need to show that $h_{V}\times_{h_X} h_{X'}$ is representable for every morphism $X'\to X$ in $\lFan/B$.
This holds since $h_{V}\times_{h_X}h_{X'}\simeq h_{V\times_X X'}$ by Proposition \ref{div.13}.
\end{proof}

\begin{lemma}
\label{equiv.26}
Suppose that $\cP$ is a class of morphisms in $\lFan/B$ closed under pullbacks.
Let $f\colon h_Y\to h_X$ be a representable $\cP$-morphism in $\Shv_{dZar}(\lFan/B)$, where $X,Y\in \lFan/B$.
Then there exists a commutative square
\begin{equation}
\begin{tikzcd}
h_{V}\ar[d,"\simeq"']\ar[r,"h_{g}"]&
h_{U}\ar[d,"h_{u}"',"\simeq"]
\\
h_{Y}\ar[r,"f"]&
h_X
\end{tikzcd}
\end{equation}
with vertical isomorphisms such that $g$ is a $\cP$-morphism and $u$ is a dividing cover in $\lFan/B$.
\end{lemma}
\begin{proof}
From \eqref{equiv.6.1}, we have a commutative square
\[
\begin{tikzcd}
h_{Y'}\ar[r,"h_{f'}"]\ar[d,"\simeq"']&
h_{X'}\ar[d,"\simeq"]
\\
h_Y\ar[r,"f"]&
h_X
\end{tikzcd}
\]
with vertical isomorphisms such that $f'$ is a $\cP$-morphism in $\lFan/B$.
Apply Lemma \ref{equiv.24} to $h_X\xrightarrow{\simeq} h_{X'}$ to obtain a commutative triangle
\[
\begin{tikzcd}
h_{X'}\ar[d,"\simeq"']\ar[r,leftarrow,"h_{u'}"]&
h_U\ar[ld,"h_u","\simeq"']
\\
h_X
\end{tikzcd}
\]
such that $u$ is a dividing cover and $u'$ is a morphism in $\lFan/B$.
Take $V:=Y'\times_{X'} U$ to conclude.
\end{proof}

\begin{lemma}
\label{equiv.25}
Let $f\colon Y\to X$ be a morphism in $\lFan/B$.
If $h_f$ is an isomorphism, then there exist dividing covers $u\colon V\to X$ and $v\colon V\to Y$ in $\lFan/B$ such that $fv=u$.
\end{lemma}
\begin{proof}
Apply Proposition \ref{div.5} to $h_f^{-1}\in h_Y(X)$ to obtain a morphism $p\colon X'\to Y$ such that $fp$ is a dividing cover.
Since $h_p$ is an isomorphism, we can also apply Proposition \ref{div.5} to $h_p^{-1}\in h_{X'}(Y)$ to obtain a morphism $q\colon Y'\to X'$ such that $pq$ is a dividing cover.
We set $V:=X'\times_Y Y'$, which is a dividing cover over $X$.
Use $Y'\times_Y Y'\simeq Y'$ and $V\times_X V\simeq V$ to obtain an induced commutative diagram
\[
\begin{tikzcd}[column sep=small, row sep=small]
Y'\times_X V\ar[r]\ar[d]&
V\ar[r]\ar[d]&
Y'\times_X V\ar[r]\ar[d]&
V\ar[d]
\\
Y'\ar[r]\ar[d]&
V\ar[r]\ar[d]&
Y'\ar[r,"fpq"]\ar[d,"pq"]&
X
\\
Y'\ar[r,"q"]&
X'\ar[r,"p"]&
Y
\end{tikzcd}
\]
whose small squares are cartesian and vertical morphisms are dividing covers.
Let $a\colon V\to Y'\times_X V$ be the graph morphism, and let $b\colon Y'\times_X V\to V$ be the projection.
From the upper row of the diagram, we see that $a$ is an inverse of $b$.
Hence $b$ is an isomorphism, so we obtain the desired dividing covers $V\to X$ and $V\to Y$.
\end{proof}

\begin{lemma}
\label{equiv.28}
Let $f\colon h_Y\to h_X$ be an isomorphism in $\Shv_{dZar}(\lFan/B)$, where $X,Y\in \lFan/B$.
Then there exist dividing covers $u\colon V\to X$ and  $v\colon V\to Y$ in $\lFan/B$ such that $fh_v=h_u$.
\end{lemma}
\begin{proof}
Lemma \ref{equiv.24} yields a dividing cover $q\colon Y'\to Y$ and a morphism $p\colon Y'\to X$ such that $fh_q=h_p$.
Since $h_p$ is an isomorphism, Lemma \ref{equiv.25} yields a dividing cover $V\to Y'$ such that the composition $V\to Y'\xrightarrow{p} X$ is a dividing cover.
The composition $V\to Y'\xrightarrow{q} Y$ is a dividing cover too.
\end{proof}

\begin{lemma}
\label{equiv.29}
Let $f\colon Y\to X$ be a $\cP$-morphism in $\lFan/B$, where $\cP$ is a class of morphisms in $\lFan/B$ closed under pullbacks.
Then $h_f$ is a representable $\cP$-morphism.
\end{lemma}
\begin{proof}
Let $h_{V}\to h_X$ be a morphism in $\Shv_{dZar}(\lFan/B)$ with $V\in \lFan/B$.
By Lemma \ref{equiv.24}, we can replace $V$ with its suitable dividing cover to assume that $h_V\to h_X$ is equal to $h_g$ for some morphism $g\colon V\to X$ in $\lFan/B$.
To conclude, observe that we have an isomorphism $h_Y\times_{h_X} h_V \simeq h_{Y\times_X V}$.
\end{proof}

For a class of morphisms $\cP$ in $\lFan/B$, we will frequently assume the following condition:
\begin{enumerate}
\item[(Div)] If $Y\to X$ is a $\cP$-morphism and $Y'\to Y$ is a dividing cover in $\lFan/B$, then there exists a dividing cover $X'\to X$ in $\lFan/B$ such that the projection $Y'\times_X X'\to X'$ is a $\cP$-morphism.
\end{enumerate}

\begin{example}
\label{equiv.16}
Suppose that $\cP$ is a class of morphisms in $\lFan/B$ closed under pullbacks.
Let $Y\to X$ be a $\cP$-morphism, and let $Y'\to Y$ be a dividing cover in $\lFan/B$.
\\[5pt]
(1) Assume that every morphism in $\cP$ is strict.
If $X$ has a fan chart $\Sigma$,
then $Y$ has a fan chart $\Sigma$.
Proposition \ref{equiv.23} yields a subdivision $\Sigma'$ of $\Sigma$ such that the pullback $Y'\times_{\T_\Sigma}\T_{\Sigma'}\to Y\times_{\T_\Sigma}\T_{\Sigma'}$ is an isomorphism.
This means that the pullback $Y'\times_X X_{\Sigma'}\to Y\times_X X_{\Sigma'}$ is an isomorphism,
where $X_{\Sigma'}:=X\times_{\T_\Sigma}\T_{\Sigma'}$.
Hence $\cP$ satisfies (Div).
\\[5pt]
(2) Let exact $\cP$ be the subclass of morphisms in $\cP$ that are exact.
Suppose that exact $\cP$ satisfies (Div).
By \cite[Proposition 4.2.3]{logA1}, there exists a dividing cover $X''\to X$ such that the projection $Y\times_X X''\to X''$ is exact $\cP$.
Our assumption on exact $\cP$ yields a dividing cover $X'\to X''$ such that the projection $Y'\times_{X} X'\to X'$ is exact $\cP$.
Hence $\cP$ satisfies (Div).
\\[5pt]
(3) Suppose that $\cP$ contains all dividing covers and is closed under compositions.
Observe that $\cP$ satisfies (Div).
If $V\to U$ is an exact $\cP$-morphism and $V'\to V$ is a dividing cover, then there exists a dividing cover $U'\to U$ such that the projection $p\colon V'\times_U U'\to U'$ is exact by \cite[Proposition 4.2.3]{logA1}.
It follows that $p$ is an exact $\cP$-morphism, so exact $\cP$ satisfies (Div).
\end{example}

According to \cite[Proposition 4.2.3]{logA1}, $f$ is a representable $\cP$-morphism if and only if $f$ is a representable exact $\cP$-morphism.

In the following cases of $\cP$, all exact $\cP$-morphisms are strict by Proposition \ref{equiv.11} so that we can use (1) and (2):

\begin{center}
\begin{tabular}{|c|c|c|}
\hline
$\cP$ & exact $\cP$
\\
\hline
dividing cover & isomorphism
\\
dividing Zariski cover & Zariski cover
\\
proper monomorphism & strict closed immersion
\\
log \'etale monomorphism & open immersion
\\
strict immersion & strict immersion
\\
\hline
\end{tabular}
\end{center}

If $\cP$ is log smooth or log \'etale, then we can use (3).

\begin{example}
A morphism of fs log schemes $i\colon Z\to X$ is called a \emph{closed immersion} if the underlying morphism of schemes $\ul{i}$ is a closed immersion and the induced morphism of structure sheaves of monoids $i_{\log}^*\cM_X\to \cM_Z$ is surjective, see \cite[Definition III.2.3.1]{Ogu}.
In this case, the construction of the fiber products in the proof of \cite[Proposition III.2.1.2]{Ogu} shows that the diagonal morphism $Z\to Z\times_X Z$ is an isomorphism, i.e., $i$ is a monomorphism.
Hence we have the following relations:
\[
\text{(strict closed immersion) }
\subset
\text{ (closed immersion) }
\subset
\text{ (proper monomorphism).}
\]
We deduce that a morphism $f$ in $\Shv_{dZar}(\lFan/B)$ is a representable closed immersion if and only if $f$ is a representable strict closed immersion.
\end{example}

With the condition \textup{(Div)}, we have a more structured version of Lemma \ref{equiv.26} as follows.

\begin{lemma}
\label{equiv.27}
Suppose that $\cP$ is a class of morphisms in $\lFan/B$ closed under pullbacks and satisfying \textup{(Div)}.
Let $f\colon h_Y\to h_X$ be a representable $\cP$-morphism in $\Shv_{dZar}(\lFan/B)$, where $X,Y\in \lFan/B$.
Then there exists a commutative square
\begin{equation}
\begin{tikzcd}
h_{V}\ar[d,"h_{v}","\simeq"']\ar[r,"h_{g}"]&
h_{U}\ar[d,"h_{u}"',"\simeq"]
\\
h_{Y}\ar[r,"f"]&
h_X
\end{tikzcd}
\end{equation}
with vertical isomorphisms such that $g$ is a $\cP$-morphism and $u$ and $v$ are dividing covers in $\lFan/B$.
\end{lemma}
\begin{proof}
By Lemmas \ref{equiv.26} and \ref{equiv.28}, there exists a commutative diagram
\[
\begin{tikzcd}
h_{Y''}\ar[d,"\simeq"',"h_{q'}"]\ar[dd,bend right=70,"h_{q}"']
\\
h_{Y'}\ar[d,"\simeq"']\ar[r,"h_{f'}"]&
h_{X'}\ar[d,"h_{p}"',"\simeq"]
\\
h_{Y}\ar[r,"f"]&
h_X
\end{tikzcd}
\]
with vertical isomorphisms such that $f'$ is a $\cP$-morphism and $p$, $q$, and $q'$ are dividing covers in $\lFan/B$.
Use (Div) to obtain a cartesian square
\[
\begin{tikzcd}
Y''\times_{X'}X''\ar[r,"f''"]\ar[d]&
X''\ar[d,"p'"]
\\
Y''\ar[r,"f'q'"]&
X'
\end{tikzcd}
\]
such that $p'$ is a dividing cover and $f''$ is a $\cP$-morphism.
Take $U:=X''$ and $V:=Y''\times_{X'}X''$ to obtain the desired commutative diagram.
\end{proof}

\begin{lemma}
\label{equiv.18}
Suppose that $\cP$ and $\cQ$ are classes of morphisms in $\lFan/B$ closed under pullbacks.
Let $h_Y\to h_X$ be a representable $\cP$-morphism, and let $h_Z\to h_Y$ be a representable $\cQ$-morphism with $X,Y,Z\in \lFan/B$.
If $\cP$ satisfies \textup{(Div)}, then there exists a commutative diagram
\[
\begin{tikzcd}
h_{W}\ar[r,"h_g"]\ar[d,"\simeq"']&
h_V\ar[d,"\simeq"]\ar[r,"h_f"]&
h_U\ar[d,"h_u"',"\simeq"]
\\
h_Z\ar[r]&
h_Y\ar[r]&
h_X
\end{tikzcd}
\]
with vertical isomorphisms such that $f$ is a $\cP$-morphism, $g$ is a $\cQ$-morphism, and $u$ is a dividing cover in $\lFan/B$.
\end{lemma}
\begin{proof}
Use Lemma \ref{equiv.26} twice to obtain a commutative diagram
\[
\begin{tikzcd}
h_{Z''}\ar[dd,"\simeq"']\ar[r,"h_b"]&
h_{Y''}\ar[d,"\simeq","h_q"']
\\
&
h_{Y'}\ar[r,"h_a"]\ar[d,"\simeq"]&
h_{X'}\ar[d,"\simeq","h_p"']
\\
h_Z\ar[r]&
h_Y\ar[r]&
h_X
\end{tikzcd}
\]
with vertical isomorphisms such that $a$ is a $\cP$-morphism, $b$ is a $\cQ$-morphism, and $p$ and $q$ are dividing covers in $\lFan/B$.
By \textup{(Div)}, there exists a dividing cover $X''\to X'$ such that the projection $Y''\times_{X'} X''\to X''$ is a $\cP$-morphism.
Take $U:=X''$, $V:=Y''\times_{X'} X''$, and $W:=Z''\times_{X'}X''$ to conclude.
\end{proof}

\begin{proposition}
\label{equiv.15}
Suppose that $\cP$ is a class of morphisms in $\lFan/B$ closed under compositions and pullbacks.
If $\cP$ satisfies \textup{(Div)}, then the class of representable $\cP$-morphisms in $\Shv_{dZar}(\lFan/B)$ is closed under compositions.
\end{proposition}
\begin{proof}
An immediate consequence of Lemma \ref{equiv.18} when $\cQ:=\cP$.
\end{proof}

\begin{proposition}
\label{equiv.37}
Suppose that $\cP$ is the class of all morphisms in $\lFan/B$.
Then a morphism $f\colon \cG\to \cF$ in $\Shv_{dZar}(\lFan/B)$ is representable if and only if $f$ is a representable $\cP$-morphism.
\end{proposition}
\begin{proof}
The if direction is trivial.
For the only if direction, assume that $f$ is representable.
Let $h_X\to \cF$ be a morphism with $X\in \lFan/B$.
Then there exists a cartesian square
\[
\begin{tikzcd}
h_Y\ar[r]\ar[d]&
h_X\ar[d]
\\
\cG\ar[r,"f"]&
\cF
\end{tikzcd}
\]
with $Y\in \lFan/B$.
To obtain \eqref{equiv.6.1}, apply Lemma \ref{equiv.24} to $h_Y\to h_X$.
\end{proof}

\begin{proposition}
\label{equiv.36}
Let $f\colon \cG\to \cF$ be a morphism in $\Shv_{dZar}(\lFan/B)$.
If $f$ is a representable Zariski cover, then $f$ is an epimorphism of sheaves.
\end{proposition}
\begin{proof}
Suppose $a\in \cF(X)$ with $X\in \lFan/B$, which can be expressed as a morphism of sheaves $a\colon h_X\to \cF$.
Lemma \ref{equiv.26} yields a commutative square
\[
\begin{tikzcd}
h_V\ar[d,"\simeq"']\ar[r,"h_g"]&
h_U\ar[d,"h_u"',"\simeq"]\\
\cG\times_{\cF} h_X\ar[r]&
h_X
\end{tikzcd}
\]
with vertical isomorphisms such that $g$ is a Zariski cover and $u$ is a dividing cover in $\lFan/B$.
Hence $(ug)^*a\in \cF(V)$ is in the image of $f(V)\colon \cG(V)\to \cF(V)$.
To conclude, observe that $V\to X$ is a dividing Zariski cover.
\end{proof}

\begin{lemma}
\label{equiv.17}
Suppose that $\cP$ is one of the following classes:
\begin{enumerate}
\item[\textup{(i)}] isomorphisms,
\item[\textup{(ii)}] strict immersions,
\item[\textup{(iii)}] open immersions,
\item[\textup{(iv)}] strict closed immersions,
\item[\textup{(v)}] Zariski covers.
\end{enumerate}
Let
\begin{equation}
\label{equiv.17.1}
\begin{tikzcd}
\cG'\ar[d,"u'"']\ar[r,"v'"]&
\cG\ar[d,"u"]
\\
\cF'\ar[r,"v"]&
\cF
\end{tikzcd}
\end{equation}
be a cartesian square in $\Shv_{dZar}(\lFan/B)$ such that $u$ is a representable Zariski cover and $v'$ is a representable $\cP$-morphism.
Then $v$ is a representable $\cP$-morphism.
\end{lemma}
\begin{proof}
We only need to consider the case when $\cF=h_X$ with $X\in \lFan/B$.
By Lemma \ref{equiv.18}, we can replace $X$ with its suitable dividing cover to assume that \eqref{equiv.17.1} is isomorphic to a cartesian square
\begin{equation}
\label{equiv.17.3}
\begin{tikzcd}
h_{Y'}\ar[d,"u'"']\ar[r,"h_{g'}"]&
h_{Y}\ar[d,"h_{f}"]
\\
\cF'\ar[r,"v"]&
h_{X}
\end{tikzcd}
\end{equation}
for some Zariski cover $f$ and $\cP$-morphism $g'$ in $\lFan/B$.

We have isomorphisms
$
h_{Y'\times_X Y} \simeq h_{Y'}\times_{\cF'}h_{Y'} \simeq h_{Y\times_X Y'}
$.
Let $c\colon h_{Y'\times_X Y}\to h_{Y\times_X Y'}$ be the composition.
Lemma \ref{equiv.28} yields dividing covers $a\colon V\to Y'\times_X Y$ and $b\colon V \to Y\times_X Y'$ such that $ch_a=h_b$.
The induced square
\[
\begin{tikzcd}
h_V\ar[r,"h_a"]\ar[d,"h_b"']&
h_{Y'\times_X Y}\ar[d]
\\
h_{Y\times_X Y'}\ar[r]&
h_{Y\times_X Y}
\end{tikzcd}
\]
commutes.
By Proposition \ref{div.5}, after replacing $V$ by its suitable dividing cover, we may assume that the induced square
\[
\begin{tikzcd}
V\ar[r,"a"]\ar[d,"b"']&
Y'\times_X Y\ar[d]
\\
Y\times_X Y'\ar[r]&
Y\times_X Y
\end{tikzcd}
\]
commutes.
Let $g''\colon V\to Y\times_X Y$ be the composition obtained from this.
We also have the induced diagram
\[
\begin{tikzcd}
h_{V}\ar[d,shift left=0.5ex,"h_{q_2b}"]\ar[d,shift right=0.5ex,"h_{q_1a}"']\ar[r,"h_{g''}"]&
h_{Y\times_X Y}\ar[d,shift left=0.5ex,"h_{p_2}"]\ar[d,shift right=0.5ex,"h_{p_1}"']
\\
h_{Y'}\ar[r,"g'"]&
h_{Y},
\end{tikzcd}
\]
where $p_1$ and $q_1$ are the first projections and $p_2$ and $q_2$ are the second projections.
The two squares formed with the left vertical or right vertical morphisms commute.
By Proposition \ref{div.5} again, after replacing $V$ by its suitable dividing cover, we may assume that the two squares in the diagram
\begin{equation}
\label{equiv.17.4}
\begin{tikzcd}
V\ar[d,shift left=0.5ex,"q_2b"]\ar[d,shift right=0.5ex,"q_1a"']\ar[r,"g''"]&
Y\times_X Y\ar[d,shift left=0.5ex,"p_2"]\ar[d,shift right=0.5ex,"p_1"']
\\
Y'\ar[r,"g'"]&
Y
\end{tikzcd}
\end{equation}
commute.
Using \cite[Proposition 4.2.3]{logA1}, we can replace $X$ with its suitable dividing cover and $Y$, $Y'$, and $V$ by their corresponding pullbacks to assume that the composition $V\to X$ is exact.
In this case, $a$ and $b$ are exact since $Y'\times_X Y$ and $Y\times_X Y'$ are strict over $X$.
Proposition \ref{equiv.11}(3) shows that $a$ and $b$ are isomorphisms.
It follows that the two squares in \eqref{equiv.17.4} are cartesian.
All the morphisms in \eqref{equiv.17.4} are strict.
By gluing, we obtain $X'\in \lSch/B$ with a cartesian square
\[
\begin{tikzcd}
Y'\ar[d,"f'"']\ar[r,"g'"]&
Y\ar[d,"f"]
\\
X'\ar[r,"g"]&
X
\end{tikzcd}
\]
such that $g$ is a $\cP$-morphism.
Since $g$ is strict,
we have $X'\in \lFan/B$.

Proposition \ref{equiv.36} shows that $h_f$ is an epimorphism of sheaves.
Hence $h_f$ is a universal effective epimorphism of sheaves.
Its pullbacks $h_{f'}$ and $u'$ are universal effective epimorphisms of sheaves too, so we have isomorphisms
\[
h_{X'}\simeq \Coeq(h_{V}\rightrightarrows h_{Y'}) \simeq \cF'.
\]
By Lemma \ref{equiv.29} for $g$, we deduce that $v$ is a representable $\cP$-morphism.
\end{proof}

\section{Divided log spaces}\label{divspace}

We adapt the definition of algebraic spaces \cite[Definition II.1.1]{zbMATH03350017} to the dividing Zariski topology as follows.

\begin{definition}
\label{equiv.5}
We say that $\cX\in \Shv_{dZar}(\lFan/B)$ is a \emph{(noetherian) divided log space over $B$} if the following two conditions are satisfied.
\begin{enumerate}
\item[(i)]
The diagonal morphism $\Delta\colon\cX\to \cX\times \cX$ is a representable strict immersion.
\item[(ii)]
There exists a representable Zariski cover $h_X\to \cX$ with $X\in \lFan/B$.
\end{enumerate}
A \emph{morphism of divided log spaces over $B$} is a morphism of sheaves.
The category of divided log spaces over $B$ is denoted by $\lSpc/B$.
\end{definition}

\begin{remark}
\label{equiv.7}
We chose the terminology ``divided'' because a dividing sheaf has no more nontrivial dividing cover, i.e., dividing is finished.
\end{remark}

\begin{proposition}
\label{equiv.20}
Let $h_Y\to \cX$ and $h_Z\to \cX$ be morphisms in $\lSpc/B$ with $Y,Z\in \lFan/B$.
Then the fiber product $h_Y\times_{\cX}h_Z$ is representable.
\end{proposition}
\begin{proof}
Consider the induced cartesian square
\[
\begin{tikzcd}
h_Y\times_{\cX}h_Z\ar[r]\ar[d]&
h_Y\times h_Z\ar[d]
\\
\cX\ar[r,"\Delta"]&
\cX\times \cX.
\end{tikzcd}
\]
The condition that $\Delta$ is a representable strict immersion implies the claim.
\end{proof}

\begin{proposition}
\label{equiv.38}
Let $\cZ\xrightarrow{g}\cY\xrightarrow{f}\cX$ be morphisms in $\lSpc/B$.
If $f$ is a representable strict (resp.\ strict separated) morphism and $fg$ is a representable strict (resp.\ strict closed) immersion, then $g$ is a representable strict (resp.\ strict closed) immersion.
\end{proposition}
\begin{proof}
Choose a representable Zariski cover $h_U\to \cX$ with $U\in \lFan/B$.
Since $f$ and $fg$ are representable, there exists a commutative diagram
\[
\begin{tikzcd}
h_W\ar[d]\ar[r]&
h_V\ar[d]\ar[r]&
h_U\ar[d]
\\
\cZ\ar[r,"g"]&
\cY\ar[r,"f"]&
\cX
\end{tikzcd}
\]
with cartesian squares.
By Lemma \ref{equiv.18} and Proposition \ref{equiv.37}, we may assume that $h_V\to h_U$ is equal to $h_u$ for some strict (resp.\ strict separated) morphism $u\colon V\to U$ in $\lFan/B$ and $h_W\to h_V$ is equal to $h_v$ for some morphism $v\colon W\to V$ in $\lFan/B$.
Lemma \ref{equiv.27} yields a commutative square
\[
\begin{tikzcd}
h_{W'}\ar[d,"h_q","\simeq"']\ar[r,"h_a"]&
h_{U'}\ar[d,"h_p"',"\simeq"]
\\
h_W\ar[r,"h_{uv}"]&
h_U
\end{tikzcd}
\]
with vertical isomorphisms such that $a$ is a strict (resp.\ strict closed) immersion and $p$ and $q$ are dividing covers in $\lFan/B$.
By Proposition \ref{div.5}, there exists a dividing cover $r\colon W''\to W'$ such that the square
\[
\begin{tikzcd}
W''\ar[r,"ar"]\ar[d,"qr"']&
U'\ar[d,"p"]
\\
W\ar[r,"uv"]&
U
\end{tikzcd}
\]
commutes.
Use \cite[Proposition 4.2.3]{logA1} to find a dividing cover $U''\to U'$ such that the projection $W''\times_{U'}U''\to U''$ is exact.
The projection $W'\times_{U'}U''\to U''$ is a strict (resp.\ strict closed) immersion, so the pullback $W''\times_{U'}U''\to W'\times_{U'} U''$ is an exact dividing cover, i.e., an isomorphism by Proposition \ref{equiv.11}(3).
It follows that the projection $W''\times_{U'}U''\to U''$ is a strict 
(resp.\ strict closed) immersion.

We have the commutative diagram
\[
\begin{tikzcd}
W''\times_{U'} U''\ar[d]\ar[r]&
V\times_U U''\ar[r]\ar[d]&
U''\ar[d]
\\
W''\ar[r]&
V\times_U U'\ar[r]&
U'.
\end{tikzcd}
\]
with cartesian squares.
Since $u$ is a strict (resp.\ strict separated) morphism, the projection $V\times_U U''\to U''$ is a strict (resp.\ strict separated) morphism.
It follows that the induced morphism $W''\times_{U'}U''\to V\times_U U''$ is a strict (resp.\ strict closed) immersion.
Since the composition $h_{V\times_U U''}\to \cY$ is a representable Zariski cover, Lemma \ref{equiv.17} shows that $g$ is a representable strict (resp.\ strict closed) immersion.
\end{proof}

\begin{proposition}
\label{equiv.39}
Let $\cY\to \cX$ be a morphism in $\lSpc/B$.
Then the diagonal morphism $\cY\to \cY\times_{\cX}\cY$ is a representable strict immersion.
\end{proposition}
\begin{proof}
Consider the induced commutative diagram
\[
\begin{tikzcd}
\cY\ar[r]&
\cY\times_{\cX}\cY\ar[r]\ar[d]&
\cY\times \cY\ar[d]
\\
&
\cX\ar[r]&
\cX\times \cX,
\end{tikzcd}
\]
where the horizontal morphisms are the diagonal morphisms, and the square is cartesian.
The diagonal morphisms $\cX\to \cX\times \cX$ and $\cY\to \cY\times \cY$ are representable strict immersions.
To finish the proof, apply Proposition \ref{equiv.38} to the upper row.
\end{proof}

\begin{proposition}
\label{equiv.21}
Let $\cY\to \cX$ and $\cZ\to \cX$ be morphisms in $\lSpc/B$.
Then the fiber product $\cY\times_{\cX}\cZ$ in $\Shv_{dZar}(\lFan/B)$ is a divided log space over $B$.
\end{proposition}
\begin{proof}
We have the induced cartesian square
\[
\begin{tikzcd}
\cY\times_{\cX}\cZ\times_{\cX}\cZ\ar[r,"d"]\ar[d]&
\cY\times_{\cX}\cZ\times \cY\times_{\cX}\cZ\ar[d]
\\
\cY\ar[r]&
\cY\times \cY.
\end{tikzcd}
\]
Since the diagonal morphism $\cY\to \cY\times \cY$ is a representable strict immersion, $d$ is a representable strict immersion.
Compose $d$ with the pulback $\cY\times_{\cX} \cZ\to \cY\times_{\cX} \cZ\times_{\cX} \cZ$ of the diagonal morphism $\cZ\to \cZ\times_{\cX}\cZ$, which is a representable strict immersion by Proposition \ref{equiv.39}, to deduce the condition (i) in Definition \ref{equiv.5} for $\cY\times_{\cX}\cZ$.

There exist representable Zariski covers $h_Y\to \cY$ and $h_Z\to \cZ$ with $Y,Z\in \lFan/B$.
The pullbacks $h_Y\times_{\cX} h_Z\to \cY\times_{\cX} h_Z$ and $\cY\times_{\cX} h_Z\to \cY\times_{\cX} \cZ$ are representable Zariski covers.
Hence the composition $h_Y\times_{\cX} h_Z\to \cY\times_{\cX} \cZ$ is a representable Zariski cover by Proposition \ref{equiv.15}.
This shows the condition (ii) in Definition \ref{equiv.5} for $\cY\times_{\cX}\cZ$.
\end{proof}

\begin{proposition}
\label{equiv.31}
Let $f\colon \cY\to \cX$ be a representable Zariski cover in $\lSpc/B$.
If $f$ is a monomorphism in $\lSpc/B$, then $f$ is an isomorphism.
\end{proposition}
\begin{proof}
The diagonal morphism $\cY\to \cY\times_{\cX} \cY$ is an isomorphism, where $\cY\times_{\cX} \cY$ is the fiber product in $\lSpc/B$, which is the fiber product in $\Shv_{dZar}(\lFan/B)$ by Proposition \ref{equiv.21}.
Hence $f$ is a monomorphism of sheaves.
By Proposition \ref{equiv.36}, $f$ is an epimorphism of sheaves.
It follows that $f$ is a stalkwise isomorphism of sheaves since the dividing Zariski topology has enough points by Remark \ref{equiv.35}, i.e., $f$ is an isomorphism.
\end{proof}

\begin{proposition}
\label{equiv.40}
Suppose that $\cP$ is the class of monomorphisms in $\lFan/B$.
If $f\colon \cY\to \cX$ be a representable $\cP$-morphism in $\lSpc/B$, then $f$ is a monomorphism in $\lSpc/B$.
\end{proposition}
\begin{proof}
Choose a representable Zariski cover $h_U\to \cX$ with $U\in \lFan/B$.
We may assume that there exists a cartesian square
\[
\begin{tikzcd}
h_V\ar[r,"h_g"]\ar[d]&
h_U\ar[d]
\\
\cY\ar[r,"f"]&
\cX
\end{tikzcd}
\]
such that $g$ is a monomorphism in $\lFan/B$.
Then we obtain a cartesian square
\[
\begin{tikzcd}
h_V\ar[r,"h_d"]\ar[d]&
h_{V\times_U V}\ar[d]
\\
\cY\ar[r,"\Delta"]&
\cY\times_{\cX} \cY,
\end{tikzcd}
\]
where $d$ and $\Delta$ are the diagonal morphisms.
Since $g$ is a monomorphism, $d$ is an isomorphism.
Use Lemma \ref{equiv.17} to show that $\Delta$ is an isomorphism, i.e., $f$ is a monomorphism.
\end{proof}

Hence any representable strict monomorphism in $\lSpc/B$ is a monomorphism.

\begin{proposition}
\label{equiv.9}
If $X\in \lSch/B$, then $h_X$ is a divided log space over $B$.
\end{proposition}
\begin{proof}
Choose a Zariski cover $\{X_i\}_{i\in I}$ of $X$ such that each $\ul{X_i}$ is affine,
and consider the union $D$ of $X_i\times_B X_i$ in $X\times_B X$ for all $i\in I$.
Let $h_V\to h_D$ be a morphism in $\lSpc/B$ with $V\in \lFan/B$.
By Proposition \ref{div.5}, after replacing $V$ by its suitable dividing cover, we may assume that this is isomorphic to $h_f$ for some morphism $f\colon V\to D$ in $\lSch/B$.
The morphism $X\to D$ induced by the diagonal morphism $X\to X\times_B X$ is a proper monomorphism, so the projection $W:=V\times_{D}X\to V$ is a proper monomorphism too.
Proposition \ref{equiv.23} yields a dividing cover $V'\to V$ such that the projection $W':=V'\times_V W\to V'$ is a strict closed immersion.
This shows that the induced morphism $h_X\to h_D$ is a representable strict closed immersion.
Since $h_D\to h_X\times h_X$ is a representable open immersion,
we deduce the condition (i) in Definition \ref{equiv.5} for $h_X$ by Proposition \ref{equiv.15}.

Choose a Zariski cover $Y\to X$ with $Y\in \lFan/B$.
Let $h_V\to h_X$ be a morphism with $V\in \lFan/B$.
By Proposition \ref{div.5}, after replacing $V$ by its suitable dividing cover, we may assume that this is isomorphic to $h_f$ for some morphism $f\colon V\to X$ in $\lSch/B$.
The projection $h_V\times_{h_X}h_Y\to h_V$ is isomorphic to $h_p$, where $p$ is the projection $V\times_X Y\to V$.
Since $p$ is a Zariski cover, $h_X$ satisfies the condition (ii) in Definition \ref{equiv.5}.
\end{proof}

Hence the essential image of the Yoneda functor \eqref{div.0.1} lies in $\lSpc/B$.

\begin{proposition}
\label{equiv.19}
Let $h_X\to \cX$ be a representable Zariski cover in $\lSpc/B$ with $X\in \lFan/B$.
Then there exists a dividing Zariski covering family $\{U_i\to X\}_{i\in I}$ with finite $I$ such that each $h_{U_i}\to \cX$ is a representable open immersion.
\end{proposition}
\begin{proof}
Use Lemmas \ref{equiv.26} and \ref{equiv.27} to obtain a commutative diagram
\[
\begin{tikzcd}
h_{Y'}\ar[d,"h_{f''}"',"\simeq"]\ar[rr,"h_{g'}"]&
&
h_{X''}\ar[d,"h_{f'}","\simeq"']
\\
h_{Y}\ar[d,"h_g"']\ar[r,"\simeq"]&
h_X\times_{\cX}h_X\ar[d,"p_1"]\ar[r,"p_2"]&
h_X\ar[d]
\\
h_{X'}\ar[r,"\simeq","h_f"']&
h_{X}\ar[r]&
\cX,
\end{tikzcd}
\]
where $p_1$ (resp.\ $p_2$) is the first (resp.\ second) projection, $f$, $f'$, and $f''$ are dividing covers, $g$ and $g'$ are Zariski covers, and $h_{Y'}\simeq h_{Y}\times_{h_X}h_{X''}$.
We can decompose $Y$ as $\amalg_{i\in I}Y_i$ such that each $Y_i\to X'$ is an open immersion.
We set $U_i:=g'(Y_i\times_Y Y')$ and $V_i:=g'^{-1}(U_i)$, and we regard them as open subschemes of $X''$ and $Y'$ respectively.
There is a cartesian square
\[
\begin{tikzcd}
h_{V_i}\ar[d]\ar[r]&
h_{U_i}\ar[d]
\\
h_{X'}\ar[r]&
\cX.
\end{tikzcd}
\]
Since $h_{X'}\to \cX$ is a representable Zariski cover and $V_i\to X'$ is a log \'etale monomorphism, Lemma \ref{equiv.17} shows that $h_{U_i}\to \cX$ is a representable open immersion.
To conclude, observe that $\amalg_{i\in I} U_i\to X''$ is a Zariski cover.
\end{proof}

\section{Zariski equivalence relations}
\label{equiv}

\'Etale equivalence relations in the theory of algebraic spaces are helpful for constructing examples.
The purpose of this section is to develop an analogous notion in the category of divided log spaces.
As an application, we explain how to glue divided log spaces.

\begin{definition}
\label{equiv.2}
Suppose $\cX\in \lSpc/B$.
A \emph{Zariski equivalence relation on $\cX$} is a morphism
\[
i\colon \cR\to \cX\times \cX
\]
in $\lSpc/B$ satisfying the following conditions.
\begin{enumerate}
\item[(i)] $i$ is a representable strict immersion.
\item[(ii)] If $p_1,p_2\colon \cX\times \cX\rightrightarrows \cX$ denote two projections, then $p_1i$ and $p_2i$ are representable Zariski covers.
\item[(iii)] For all $T\in \lFan/B$, $\cR(T)\to \cX(T)\times \cX(T)$ is an equivalence relation.
\end{enumerate}
\end{definition}

By Proposition \ref{equiv.40}, the condition (i) implies that $i$ is a monomorphism, i.e., $\cR(T)\to \cX(T)\times \cX(T)$ is injective for all $T\in \lFan/B$.

Let $\cX/\cR$ denote the dividing Zariski sheaf associated with the presheaf
\[
(T\in \lFan/B)\mapsto \cX(T)/\cR(T).
\]
There is an induced cartesian square
\begin{equation}
\label{equiv.2.1}
\begin{tikzcd}
\cR\ar[d,"p_1i"']\ar[r,"p_2i"]&
\cX\ar[d]
\\
\cX\ar[r]&
\cX/\cR.
\end{tikzcd}
\end{equation}

Suppose that $\cT$ is a Zariski equivalence relation on $\cY\in \lSpc/B$.
If there is a morphism $f\colon \cY\to \cX$ and a commutative square
\[
\begin{tikzcd}
\cT\ar[r]\ar[d]&
\cR\ar[d]
\\
\cY\times \cY\ar[r,"f\times f"]&
\cX\times \cX,
\end{tikzcd}
\]
then there is an induced morphism $\cY/\cT\to \cX/\cR$.

\begin{proposition}
\label{equiv.3}
Let $\cR$ be a Zariski equivalence relation on $\cX\in \lSpc/B$.
Then $\cX/\cR\in \lSpc/B$.
\end{proposition}
\begin{proof}
Let $h_V\to \cX/\cR\times \cX/\cR$ be a morphism with $V\in \lFan/B$.
The morphism $\cX\to \cX/\cR$ is an epimorphism, so there exists a dividing Zariski cover $V'\to V$ in $\lFan/B$ such that the composition $h_{V'}\to h_{V}\to \cX/\cR\times \cX/\cR$ factors through $\cX\times \cX$.
From the cartesian square \eqref{equiv.2.1}, we have an isomorphism
\[
h_{V'}\times_{\cX/\cR\times \cX/\cR}\cX/\cR
\simeq
h_{V'}\times_{\cX\times \cX}\cR.
\]
Since $\cR$ is a Zariski equivalence relation on $\cX$, the projection $h_{V'}\times_{\cX\times \cX}\cR\to h_{V'}$ is a representable strict immersion.

We set $\cF:=h_V\times_{\cX/\cR\times \cX/\cR}\cX/\cR$
to have a cartesian square
\[
\begin{tikzcd}
h_{V'}\times_{\cX\times \cX}\cR\ar[r]\ar[d]&
h_{V'}\ar[d]
\\
\cF\ar[r]&
h_V.
\end{tikzcd}
\]
Lemma \ref{equiv.17} shows that $\cF\to h_V$ is a representable strict immersion.
This shows that the diagonal morphism $\cX/\cR\to \cX/\cR\times \cX/\cR$ is a representable strict immersion, which verifies the axiom (i) of divided log spaces for $\cX/\cR$.

Let $h_V\to \cX/\cR$ be a morphism with $V\in \lFan/B$.
There exists a dividing Zariski cover $V'\to V$ in $\lFan/B$ such that the composition $h_{V'}\to h_{V}\to \cX/\cR$ factors through $\cX$.
Since \eqref{equiv.2.1} is cartesian, we have an isomorphism
\[
h_{V'}\times_{\cX/\cR}\cX
\simeq
h_{V'}\times_{\cX}\cR,
\]
where the morphism $\cR\to \cX$ in this formulation is $p_1i$.
Hence we have a cartesian square
\[
\begin{tikzcd}
h_{V'}\times_{\cX}\cR\ar[r]\ar[d]&
h_{V'}\ar[d]
\\
h_V\times_{\cX/\cR}\cX\ar[r]&
h_V.
\end{tikzcd}
\]
Since $p_1i$ is a representable Zariski cover, Lemma \ref{equiv.17} shows that the projection $h_V\times_{\cX/\cR}\cX\to h_V$ is a representable Zariski cover.
It follows that $\cX\to \cX/\cR$ is a representable Zariski cover.
Hence $\cX/\cR$ satisfies the axiom (ii) of divided log spaces.
\end{proof}

\begin{proposition}
Let $f\colon \cY\to \cX$ be a representable Zariski cover in $\lSpc/B$.
Then the induced morphism $i\colon \cR:=\cY\times_{\cX}\cY\to \cY\times \cY$ is a Zariski equivalence relation on $\cY$.
Furthermore, $\cY/\cR\simeq \cX$.
\end{proposition}
\begin{proof}
For $T\in \lFan/B$, a section $(a,b)\in \cY(T)\times \cY(T)$ is in $\cR(T)$ if and only if $f(a)=f(b)$ in $\cX(T)$.
This explicit description shows that $\cR(T)\to \cX(T)\times \cX(T)$ is an equivalence relation and $\cY/\cR\simeq \cX$.
Since the diagonal morphism $\cX\to \cX\times \cX$ is a representable strict immersion, so is $i$.
The assumption that $f$ is a representable Zariski cover implies that the two projections $\cR\rightrightarrows \cY$ are representable Zariski covers.
Hence $\cR$ is a Zariski equivalence relation on $\cY$.
\end{proof}

\begin{construction}
\label{equiv.13}
Let $I$ be a finite set.
Assume that we have given the gluing data
\begin{enumerate}
\item[(1)]
$\cX_i\in \lSpc/B$ for all $i\in I$,
\item[(2)]
representable open immersions $\cU_{ij}\to \cX_i$ for all $i,j\in I$,
\item[(3)]
isomorphisms $\varphi_{ij} \colon \cU_{ij}\to \cU_{ji}$ for all $i,j\in I$,
\end{enumerate}
satisfying the following conditions for $i,j,k\in I$:
\begin{enumerate}
\item[(i)]
$\cU_{ii}=\cX_i$ and $\varphi_{ii}=\id$,
\item[(ii)]
there exists an isomorphism $\psi_{ijk}\colon  \cU_{ij}\times_{\cX_i}\cU_{ik}\to \cU_{ji}\times_{\cX_j}\cU_{jk}$ such that the square
\[
\begin{tikzcd}
\cU_{ij}\times_{\cX_i}\cU_{ik}\ar[d,hookrightarrow]\ar[r,"\psi_{ijk}"]&
\cU_{ji}\times_{\cX_j}\cU_{jk}\ar[d,hookrightarrow]
\\
\cU_{ij}\ar[r,"\varphi_{ij}"]&
\cU_{ji}
\end{tikzcd}
\]
commutes,
\item[(iii)]
the diagram
\[
\begin{tikzcd}
\cU_{ij}\times_{\cX_i}\cU_{ik}\ar[d,"\simeq"']\ar[r,"\psi_{ijk}"]&
\cU_{ji}\times_{\cX_j}\cU_{jk}\ar[r,"\simeq"]&
\cU_{jk}\times_{\cX_j} \cU_{ji}\ar[d,"\psi_{jki}"]
\\
\cU_{ik}\times_{\cX_i}\cU_{ij}\ar[r,"\psi_{ikj}"]&
\cU_{ki}\times_{\cX_k}\cU_{kj}\ar[r,"\simeq"]&
\cU_{kj}\times_{\cX_k}\cU_{ki}
\end{tikzcd}
\]
commutes.
\end{enumerate}

In this setting, let us explain the gluing construction.
We set $\cX:=\amalg_{i\in I} \cX_i$ and $\cR:=\amalg_{i,j\in I}\cU_{ij}$.
The composition
\[
\cU_{ij}\xrightarrow{\Gamma_{\varphi_{ij}}} \cU_{ij}\times \cU_{ji} \hookrightarrow  \cX_i\times \cX_j
\]
induces a morphism $\cR\to \cX\times \cX$, where $\Gamma_{\varphi_{ij}}$ is the graph morphism.
Since the diagonal morphism $\cU_{ij}\to \cU_{ij}\times \cU_{ij}$ is a representable strict immersion, $\Gamma_{\varphi_{ij}}$ is a representable strict immersion.
Hence $\cR\to \cX\times \cX$ is a representable strict immersion.
The two compositions $\cR\to \cX\times \cX\rightrightarrows \cX$ are representable Zariski covers since the induced morphism $\amalg_{j\in I} \cU_{ij}\to \cX_i$ is a representable Zariski cover for all $i\in I$.
Together with the above conditions (i)--(iii), we see that $\cR$ is a Zariski equivalence relation on $\cX$.
The \emph{gluing of $\{\cX_{i}\}_{i\in I}$ along $\{\cU_{ij}\}_{i,j\in I}$} is defined to be $\cX/\cR$.

Apply Lemma \ref{equiv.17} to the induced cartesian squares
\[
\begin{tikzcd}
\amalg_{j\in I}\cU_{ij}\ar[r]\ar[d]&
\amalg_{j\in I}\cX_j\ar[d]
\\
\cX_i\ar[r]&
\cX/\cR
\end{tikzcd}
\quad
\begin{tikzcd}
\amalg_{i,j\in I}\cU_{ij}\ar[r]\ar[d]&
\amalg_{j\in I}\cX_j\ar[d]
\\
\amalg_{i\in I}\cX_i\ar[r]&
\cX/\cR
\end{tikzcd}
\]
to see that the induced morphism $\cX_i\to \cX/\cR$ is a representable open immersion for all $i\in I$ and the induced morphism $\amalg_{i\in I}\cX_i\to \cX/\cR$ is a representable Zariski cover.

For the functoriality of the gluing construction, assume that another gluing data
\[
\text{
$\cY_{i}$, $\cV_{ij}\to \cY_{i}$, $\varphi_{ij}$, and $\psi_{ijk}$
}
\]
for $i,j,k\in J$ are given, where $J$ is a finite set.
Furthermore, assume that a map $\eta\colon I\to J$, morphisms $\cX_i\to \cY_{\eta(i)}$ for all $i\in I$, and morphisms $\cU_{ij}\to \cV_{\eta(i)\eta(j)}$ are given too such that
the squares
\[
\begin{tikzcd}
\cU_{ij}\ar[r]\ar[d]&
\cX_i\ar[d]
\\
\cV_{\eta(i)\eta(j)}\ar[r]&
\cY_{\eta(i)}
\end{tikzcd}
\quad
\begin{tikzcd}
\cU_{ij}\ar[d]\ar[r,"\varphi_{ij}"]&
\cU_{ji}\ar[d]
\\
\cV_{\eta(i)\eta(j)}\ar[r,"\varphi_{\eta(i)\eta(j)}"]&
\cV_{\eta(j)\eta(i)}
\end{tikzcd}
\]
commute.
We set $\cY:=\amalg_{i\in J}\cY_{i}$ and $\cT:=\amalg_{i,j\in J}\cV_{ij}$.
There is an induced commutative square
\[
\begin{tikzcd}
\cR\ar[d]\ar[r]&
\cT\ar[d]
\\
\cX\times \cX\ar[r]&
\cY\times \cY.
\end{tikzcd}
\]
This induces a functorial morphism $\cX/\cR\to \cY/\cT$.
\end{construction}

\begin{definition}
Let $\{\cU_i\to \cX\}_{i\in I}$ be a family of representable open immersions in $\lSpc/B$ with finite $I$.
The \emph{union $\cup_{i\in I} \cU_i$ of $\{\cU_i\}_{i\in I}$} is defined to be the gluing of $\{\cU_i\}_{i\in I}$ along $\{\cU_i\times_{\cX} \cU_j\}_{i,j\in I}$.

Observe that the induced morphism $\cU_a\to \cup_{i\in I} \cU_i$ is a representable open immersion for all $a\in I$ and the induced morphism $\amalg_{i\in I} \cU_i\to \cup_{i\in I} \cU_i$ is a representable Zariski cover.
\end{definition}

\begin{proposition}
Let $\{\cU_i\to \cX\}_{i\in I}$ be a family of representable open immersions in $\lSpc/B$ with finite $I$.
Then the induced morphism $\cup_{i\in I} \cU_i\to \cX$ is a representable open immersion.
\end{proposition}
\begin{proof}
There exists a representable Zariski cover $h_X\to \cX$ with $X\in \lFan/B$.
Apply Lemma \ref{equiv.17} to the induced cartesian square
\[
\begin{tikzcd}
\cup_{i\in I} (\cU_i \times_{\cX} h_X)\ar[d]\ar[r]&
h_X\ar[d]
\\
\cup_{i\in I} \cU_i\ar[r]&
\cX
\end{tikzcd}
\]
to reduce to the case when $\cX=h_X$ with $X\in \lFan/B$.

Then by Lemma \ref{equiv.26},
there exists a commutative square
\[
\begin{tikzcd}
h_{U_i}\ar[r,"h_{u_i}"]\ar[d,"\simeq"']&
h_{X_i}\ar[d,"h_{f_i}"',"\simeq"]
\\
\cU_i\ar[r]&
h_X
\end{tikzcd}
\]
with vertical isomorphisms
such that $f_i$ is a dividing cover and $u_i$ is an open immersion in $\lFan/B$.
Let $Y$ be the fiber product of all $X_i$ over $X$,
and we set $V_i:=U_i\times_{X_i} Y$.
The induced morphism $V_i\to Y$ is an open immersion for all $i$.
Since $\cup_{i\in I} \cU_i\simeq h_{\cup_{i\in I} V_i}$ and $h_X\simeq h_Y$,
we deduce that $\cup_{i\in I} \cU_i\to \cX$ is a representable open immersion.
\end{proof}

\section{Properties of morphisms of divided log spaces}
\label{property}

When a morphism $f\colon \cY\to \cX$ and a representable Zariski cover $g\colon \cZ\to \cY$ such that $fg$ is a representable smooth morphism are given, it is natural to regard $f$ as a smooth morphism.
However, it is unclear whether $f$ is a representable smooth morphism or not.

This is the reason why we introduce a class of morphisms that can include non-representable morphisms as follows.

\begin{definition}
Let $\cP$ be a class of morphisms in $\lFan/B$ closed under pullbacks and compositions and satisfying (Div) and the following condition:
\begin{itemize}
\item[(Zarloc)]
If $f\colon Y\to X$ is a morphism and $u\colon U\to Y$ is a Zariski cover in $\lFan/B$,
then $fu\in \cP$ implies $f\in \cP$.
\end{itemize}

We say that a morphism $f\colon \cY\to \cX$ in $\lSpc/B$ is a \emph{$\cP$-morphism} if there exists a representable Zariski cover $u\colon \cU\to \cY$ such that $fu$ is a representable $\cP$-morphism.
\end{definition}

\begin{example}
By \cite[Theorem 0.2]{zbMATH06164842}, the classes of log smooth and log \'etale morphisms in $\lFan/B$ satisfy (Zarloc).
This implies that the classes of exact log smooth and Kummer \'etale morphisms in $\lFan/B$ satisfy (Zarloc).

The classes of Zariski covers, strict Nisnevich covers, and strict \'etale covers also satisfy (Zarloc).

If $f\colon Y\to X$ is a morphism and $u\colon U\to Y$ is a Zariski cover in $\lFan/B$ such that $fu$ is a monomorphism,
then $u$ is a monomorphism.
This implies that $u$ is an isomorphism.
Hence the classes of open immersions and closed immersions in $\lFan/B$ satisfy (Zarloc).
In particular,
we have the notions of open immersions and closed immersions in $\lSpc/B$,
which will be useful in Sections \ref{complement} and \ref{blow-up}.
\end{example}

\begin{proposition}
\label{property.4}
Let $\cP$ be a class of morphisms in $\lFan/B$ closed under pullbacks and compositions and satisfying \textup{(Div)} and \textup{(Zarloc)}.
Then a morphism in $\lSpc/B$ is a representable $\cP$-morphism if and only if it is representable and a $\cP$-morphism.
\end{proposition}
\begin{proof}
Any representable $\cP$-morphism in $\lSpc/B$ is obviously representable and a $\cP$-morphism.

For the converse,
assume that $f\colon Y\to X$ is a morphism in $\lFan/B$ such that $h_f$ is a $\cP$-morphism.
We need to show that $h_f$ is a representable $\cP$-morphism.
There exists a representable Zariski cover $h_U\to h_Y$ with $U\in \lFan/B$ such that the composition $h_U\to h_X$ is a representable $\cP$-morphism.
By Lemma \ref{equiv.27},
after replacing $U$ by its suitable dividing cover,
we may assume that $h_U\to h_Y$ is equal to $h_g$ for some dividing Zariski cover $g\colon U\to Y$.
Then apply Lemma \ref{equiv.27} to $h_{fg}$ to obtain a commutative square
\[
\begin{tikzcd}
h_{V}\ar[d,"h_{w'}","\simeq"']\ar[r,"h_{g'}"]&
h_{X'}\ar[d,"h_w"',"\simeq"]
\\
h_{U}\ar[r,"h_{fg}"]&
h_X
\end{tikzcd}
\]
with vertical isomorphisms such that $g'$ is a $\cP$-morphism and $w$ and $w'$ are dividing covers.
By Proposition \ref{div.5},
there exists a dividing cover $v\colon V'\to V$ such that the square
\[
\begin{tikzcd}
V'\ar[r,"g'v"]\ar[d,"w'v"']&
X'\ar[d,"w"]
\\
U\ar[r,"fg"]&
X
\end{tikzcd}
\]
commutes.
Proposition \ref{equiv.23} yields a dividing cover $Y'\to Y$ such that the projection $V'':=V'\times_Y Y'\to Y'$ is a Zariski cover.
Since $\cP$ satisfies (Div),
there exists a dividing cover $X''\to X'$ such that the projection $V''\times_X X''\simeq V''\times_{X'}X''\to X''$ is a $\cP$-morphism.
The first arrow in
\[
V''\times_{X} X''
\to
Y'\times_X X''
\to
X''
\]
is a Zariski cover.
Use (Zarloc) to see that the projection $Y'\times_X X''\to X''$ is a $\cP$-morphism.
This implies that $h_f\colon h_{Y}\to h_X$ is a representable $\cP$-morphism.
\end{proof}

\begin{proposition}
\label{property.1}
Let $\cP$ be a class of morphisms in $\lFan/B$ closed under compositions and pullbacks and satisfying \textup{(Div)} and \textup{(Zarloc)}.
For all $\cP$-morphism $\cY\to \cX$ in $\lSpc/B$,
there exists a commutative square
\[
\begin{tikzcd}
h_V\ar[r,"h_g"]\ar[d]&
h_U\ar[d]
\\
\cY\ar[r]&
\cX
\end{tikzcd}
\]
such that $h_U\to \cX$ and $h_V\to \cY\times_{\cX} h_U$ are representable Zariski covers and $g$ is a $\cP$-morphism in $\lFan/B$.
\end{proposition}
\begin{proof}
Choose a representable Zariski cover $h_X\to \cX$ with $X\in \lFan/B$.
There exists a representable Zariski cover $\cU\to \cY$ such that the composition $\cU\to \cX$ is a representable $\cP$-morphism.
By Lemma \ref{equiv.26}, there exists a commutative square
\[
\begin{tikzcd}
h_V\ar[r,"h_{g}"]\ar[d,"\simeq"']&
h_{U}\ar[d,"h_u"',"\simeq"]
\\
\cU\times_{\cX} h_X\ar[r]&
h_X
\end{tikzcd}
\]
such that $g$ is a $\cP$-morphism and $u$ is a dividing cover.
Since $\cU\to \cY$ is a representable Zariski cover,
$\cU\times_{\cX} h_X\to \cY\times_{\cX} h_X$ is a representable Zariski cover.
This means that $h_V\to \cY\times_{\cX} h_U$ is a representable Zariski cover.
\end{proof}

\begin{proposition}
\label{equiv.42}
Let $\cP$ be a class of morphisms in $\lFan/B$ closed under compositions and pullbacks and satisfying \textup{(Div)} and \textup{(Zarloc)}.
Then the class of $\cP$-morphisms in $\lSpc/B$ is closed under pullbacks and compositions.
\end{proposition}
\begin{proof}
Let $\cY\to \cX$ be a $\cP$-morphism, and let $\cX'\to \cX$ be a morphism in $\lSpc/B$.
There exists a representable Zariski cover $\cU\to \cY$ such that the composition $\cU\to \cX$ is a representable $\cP$-morphism.
The pullback $\cU\times_{\cX} \cX'\to \cY\times_{\cX}\cX'$ is a representable Zariski cover, and the projection $\cU\times_{\cX} \cX'\to \cX'$ is a representable $\cP$-morphism.
Hence the projection $\cY\times_{\cX}\cX'\to \cX'$ is a $\cP$-morphism.

Let $\cZ\to \cY\to \cX$ be $\cP$-morphisms in $\lSpc/B$.
There exists a representable Zariski cover $\cU\to \cY$ such that the composition $\cU\to \cX$ is a representable $\cP$-morphism.
By the above paragraph, the projection $\cZ\times_{\cY}\cU\to \cU$ is a $\cP$-morphism.
There exists a representable Zariski cover $\cV\to \cZ\times_{\cY}\cU$ such that the composition $\cV\to \cU$ is a representable $\cP$-morphism.
By Proposition \ref{equiv.15}, the composition $\cV\to \cX$ is a representable $\cP$-morphism, and the composition $\cV\to \cZ$ is a representable Zariski cover.
Hence $\cZ\to \cX$ is a $\cP$-morphism.
\end{proof}

\begin{proposition}
Let $\cZ\xrightarrow{g}\cY\xrightarrow{f}\cX$ be morphisms in $\lSpc/B$.
If $f$ is log \'etale and $fg$ is log \'etale (resp.\ log smooth), then $g$ is log \'etale (resp.\ log smooth).
\end{proposition}
\begin{proof}
There exists a representable Zariski cover $\cU\to \cY$ such that the composition $\cU\to \cX$ is a representable log \'etale morphism in $\lSpc/B$.
The composition $\cZ\times_{\cY}\cU\to \cX$ is log \'etale (resp.\ log smooth), so there exists a representable Zariski cover $\cV\to \cZ\times_{\cY}\cU$ such that the composition $\cV\to \cX$ is representable log \'etale (resp. log smooth).

By Lemma \ref{equiv.18}, there exists a commutative square
\[
\begin{tikzcd}
h_V\ar[d]\ar[r,"h_v"]&
h_U\ar[d]\ar[r,"h_u"]&
h_X\ar[d]
\\
\cV\ar[r]&
\cU\ar[r]&
\cX
\end{tikzcd}
\]
with cartesian squares such that the vertical morphisms are representable Zariski covers, $u$ is a log \'etale morphism, and $v$ is a morphism in $\lFan/B$.
The composition $h_V\to h_X$ is a representable log \'etale (resp.\ log smooth) morphism, so Lemma \ref{equiv.27} yields a dividing cover $v'\colon V'\to V$ and a log \'etale (resp.\ log smooth) morphism $p\colon V'\to X$ such that $h_{uvv'}=h_p$.
By Proposition \ref{div.5}, we can replace $V'$ with its suitable dividing cover to assume that the diagram
\[
\begin{tikzcd}
V'\ar[d,"v'"']\ar[rrd,"p",bend left=20]
\\
V\ar[r,"v"]&
U\ar[r,"u"]&
X
\end{tikzcd}
\]
commutes.
Owing to \cite[Remark IV.3.1.2]{Ogu},
the composition $V'\to U$ is log \'etale (resp.\ log smooth).
Hence the composition $h_{V'}\to \cY$ is a representable log \'etale (resp.\ log smooth) morphism too.
To conclude, observe that the composition $h_{V'}\to \cZ$ is a representable Zariski cover.
\end{proof}

\begin{proposition}
\label{property.3}
Let $f\colon \cY\to \cX$ be a morphism in $\lSpc/B$.
If $f$ is an open immersion (resp.\ strict closed immersion),
then $f$ is a representable open immersion (resp.\ representable strict closed immersion).
\end{proposition}
\begin{proof}
There exists a representable Zariski cover $g\colon \cU\to \cY$ such that $fg$ is a representable open immersion (resp.\ representable strict closed immersion).
Proposition \ref{equiv.40} shows that $fg$ is a monomorphism in $\lSpc/B$, so $g$ is a monomorphism in $\lSpc/B$ too.
By Proposition \ref{equiv.31}, $g$ is an isomorphism.
Hence $f$ is a representable open immersion (resp.\ representable strict closed immersion).
\end{proof}

\section{Topologies on divided log spaces}
\label{topology}

In this section, we begin with introducing several topologies on $\lSpc/B$.
Then we compare the categories of sheaves on $\lFan/B$ and $\lSpc/B$.

\begin{definition}
\label{divtop.1}
Consider the following classes of morphisms in $\lFan/B$:

\begin{center}
\small
\begin{tabular}{|c|c|}
\hline
{\normalsize $\cP$} & {\normalsize exact $\cP$}
\\
\hline
dividing Zariski covers & Zariski covers
\\
dividing Nisnevich covers & strict Nisnevich covers
\\
dividing \'etale covers & strict \'etale covers 
\\
log \'etale covers & Kummer \'etale covers 
\\
\hline
\end{tabular}
\end{center}
The smallest topologies $t_{\cP}$ containing all exact $\cP$-morphism as a covering for the above cases are called the \emph{Zariski}, \emph{strict Nisnevich}, \emph{strict \'etale}, and \emph{Kummer \'etale topologies}.
We also call them as the \emph{dividing Zariski}, \emph{dividing Nisnevich}, \emph{dividing \'etale}, and \emph{log \'etale topologies}.

Let $\lSmSpc/B$ be the full subcategory of $\lSpc/B$ consisting of $\cX$ that is log smooth over $B$.
\end{definition}

\begin{proposition}
\label{divtop.2}
For the above four cases of $\cP$, the topology $t_\cP$ on $\lSpc/B$ is the smallest topology such that all representable $\cP$-morphisms are covers.
\end{proposition}
\begin{proof}
Immediate from the fact that every $\cP$-morphism admits a refinement that is a representable $\cP$-morphism.
\end{proof}

Let $\varphi\colon \cC\to \cC'$ be a functor of sites.
There is a functor
\[
\varphi^*\colon \Shv(\cC')
\to
\Shv(\cC)
\]
such that $\varphi^*\cF(X):=\cF(\varphi(X))$ for $X\in \cC$ and $\cF\in \Shv(\cC')$.
If $\varphi$ is a continuous functor of sites, then according to \cite[Proposition III.1.2]{SGA4}, $\varphi^*$ admits a left adjoint
\[
\varphi_!\colon \Shv(\cC)\to \Shv(\cC').
\]

Consider the induced commutative diagram
\begin{equation}
\begin{tikzcd}
\lSmFan/B\ar[r,"\beta"]\ar[d,hookrightarrow,"\alpha"']&
\lSm/B\ar[d,hookrightarrow,"\alpha'"]\ar[r,"\gamma"]&
\lSmSpc/B\ar[d,hookrightarrow,"\alpha''"]
\\
\lFan/B\ar[r,"\beta'"]&
\lSch/B\ar[r,"\gamma'"]&
\lSpc/B,
\end{tikzcd}
\end{equation}
where $\lSm/B$ (resp.\ $\lSmFan/B$) denotes the full subcategory of $\lSch/B$ (resp.\ $\lFan/B$) consisting of log smooth fs log schemes.
Let $\cP$ be one of the four class of morphisms in Definition \ref{divtop.1}.
These functors are continuous functors of sites for the $t_\cP$-topology, and hence we have a commutative diagram
\begin{equation}
\begin{tikzcd}
\Shv_{t_{\cP}}(\lSmFan/B)\ar[r,"\beta_!"]\ar[d,"\alpha_!"']&
\Shv_{t_{\cP}}(\lSm/B)\ar[d,"\alpha_!'"]\ar[r,"\gamma_!"]&
\Shv_{t_{\cP}}(\lSmSpc/B)\ar[d,"\alpha_!''"]
\\
\Shv_{t_{\cP}}(\lFan/B)\ar[r,"\beta_!'"]&
\Shv_{t_{\cP}}(\lSch/B)\ar[r,"\gamma_!'"]&
\Shv_{t_{\cP}}(\lSpc/B).
\end{tikzcd}
\end{equation}
Due to the implication (i)$\Rightarrow$(ii) in \cite[Th\'eor\`eme III.4.1]{SGA4}, $\beta_!$ and $\beta_!'$ are equivalences.
Since $\alpha$, $\alpha'$, and $\alpha''$ are cocontinuous and fully faithful,  \cite[Proposition III.2.6]{SGA4} shows that $\alpha_!$, $\alpha_!'$, and $\alpha_!''$ are fully faithful.

If $X\in \lFan/B$ and $h_Y\to h_X$ is a representable $t_\cP$-cover with $Y\in \lFan/B$, then Lemma \ref{equiv.26} yields a commutative square
\[
\begin{tikzcd}
h_{Y'}\ar[d,"\simeq"']\ar[r,"h_{f'}"]&
h_{X'}\ar[d,"h_g"',"\simeq"]
\\
h_Y\ar[r]&
h_X
\end{tikzcd}
\]
with vertical isomorphisms such that $g$ is a dividing cover and $f'$ is a $t_\cP$-cover.
The composition $Y'\to X$ is a $t_{\cP}$-cover and $h_{Y'}\to h_X$ refines $h_Y\to h_X$.
This shows that $\gamma'\beta'$ is cocontinuous.
We can similarly show that $\gamma\beta$ is cocontinuous.

\begin{proposition}
\label{sheaves.1}
Let $\cP$ be as above.
The functors
\begin{gather*}
\gamma_!\colon \Shv_{t_{\cP}}(\lSm/B)\to \Shv_{t_{\cP}}(\lSmSpc/B),
\\
\gamma_!'\colon \Shv_{t_{\cP}}(\lSch/B)\to \Shv_{t_{\cP}}(\lSpc/B)
\end{gather*}
are equivalences.
\end{proposition}
\begin{proof}
By the above observation, we only need to show that $\gamma_!\beta_!$ and $\gamma_!'\beta_!'$ are equivalences.
We focus on $\gamma_!'\beta_!'$ since the proofs are similar.

Let us check the conditions (1)--(5) in \cite[Tag 03A0]{Stacks} for $\gamma'\beta'$.
We have checked the conditions (1) and (2) above.
The conditions (3) and (4) are consequences of Proposition \ref{div.5}.
To show the condition (5), consider a representable Zariski cover $h_X\to \cX$ with $X\in \lFan/B$.

Hence we have checked the conditions (1)--(5), and we deduce that $\gamma_!'\beta_!'$ is an equivalence.
\end{proof}

\begin{definition}
\label{divtop.3}
A \emph{Zariski distinguished square in $\lSpc/B$} is a cartesian square in $\lSpc/B$ of the form
\begin{equation}
\label{divtop.3.1}
\begin{tikzcd}
\cW\ar[d,"g'"']\ar[r,"f'"]&
\cV\ar[d,"g"]
\\
\cU\ar[r,"f"]&
\cX
\end{tikzcd}
\end{equation}
such that $f$ and $g$ are representable open immersions and the induced morphism $\cU\amalg \cV\to \cX$ is a representable Zariski cover.
The \emph{Zariski cd-structure on $\lSpc/B$} is the collection of Zariski distinguished squares.

By Proposition \ref{equiv.40},
$f$ and $g$ are monomorphisms.
Observe that the square
\[
\begin{tikzcd}
\cW\ar[d,"\Delta"']\ar[r,"g"]&
\cV\ar[d,"\Delta"]
\\
\cW\times_{\cU} \cW\ar[r,"g'\times_g g'"]&
\cV\times_{\cX} \cV
\end{tikzcd}
\]
whose vertical morphisms are the diagonal morphisms is a Zariski distinguished square since the vertical morphisms are isomorphisms.
The Zariski cd-structure on $\lSpc/B$ is complete and regular in the sense of \cite[Definitions 2.3, 2.10]{Voe10a}.

The topology associated with the Zariski cd-structure is defined to be the smallest Grothendieck topology containing $\cU\amalg \cV\to \cX$ as a covering for all distinguished square of the form \eqref{divtop.3.1}.
\end{definition}

\begin{proposition}
The Zariski topology on $\lSpc/B$ is the topology associated with the Zariski cd-structure on $\lSpc/B$.
\end{proposition}
\begin{proof}
Let $\cY\to \cX$ be a Zariski cover in $\lSpc/B$.
There exists a representable Zariski cover $\cU\to \cY$ such that the composition $\cU\to \cX$ is a representable Zariski cover.
By Proposition \ref{equiv.19},
we may assume $\cU\simeq \amalg_{i\in I} h_{U_i}$ with finite $I$ and each morphism $h_{U_i}\to \cX$ is a representable open immersion,
where $U_i\in \lFan/B$ for all $i\in I$.

Let $t$ be the topology associated with the Zariski cd-structure.
The sieve generated by $\{h_{U_a},\cup_{i\in I-\{a\}} h_{U_i}\to \cX\}$ is a $t$-covering sieve for all $a\in I$.
By induction on the number of elements of $I$,
we see that the sieve generated by $\{h_{U_i}\to \cX\}_{i\in I}$ is a $t$-covering sieve.
It follows that the sieve generated by $\cY\to \cX$ is a $t$-covering sieve.
\end{proof}

\section{Open complements}
\label{complement}

In this section, we define the open complement of a closed immersion of divided log spaces as a universal property.
We also show that the open complements always exist and are compatible with pullbacks.

\begin{definition}
\label{comp.1}
For a closed immersion $\cZ\to \cX$ in $\lSpc/B$, the \emph{open complement of $\cZ$ in $\cX$}, denoted $\cX-\cZ$, is defined to be a final object (if exists) of the full subcategory of $\lSpc/\cX$ consisting of morphisms $\cY\to \cX$ such that $\cZ\times_{\cX}\cY=\emptyset$.
\end{definition}

\begin{lemma}
\label{comp.2}
Let $\cZ\to \cX$ be a closed immersion in $\lSpc/B$, and let $\cX'\to \cX$ be a morphism in $\lSpc/B$.
We set $\cZ':=\cZ\times_{\cX} \cX'$.
If $\cX-\cZ$ exists, then $\cX'-\cZ'$ exists, and there is an isomorphism
\[
\cX'-\cZ'
\simeq
(\cX-\cZ)\times_{\cX}\cX'.
\]
\end{lemma}
\begin{proof}
Suppose that $\cY\to \cX'$ is a morphism in $\lSpc/B$ such that $\cZ'\times_{\cX'}\cY=\emptyset$.
Then $\cZ\times_{\cX}\cY=\emptyset$, so there exists a unique morphism $\cY\to \cX-\cZ$ over $\cX$.
This means that there exists a unique morphism 
$\cY\to (\cX-\cZ)\times_{\cX}\cX'$ over $\cX'$, which completes the proof.
\end{proof}

\begin{lemma}
\label{comp.3}
Let $\cZ\to \cX$ be a closed immersion in $\lSpc/B$, and let $\amalg_{i\in I} \cX_i\to \cX$ be a representable Zariski cover with finite $I$ such that each $\cX_i\to \cX$ is a representable open immersion.
We set $\cX_{ij}:=\cX_i\times_{\cX}\cX_j$, $\cZ_i:=\cZ\times_{\cX} \cX_i$, and $\cZ_{ij}:=\cZ\times_{\cX} \cX_{ij}$ for $i,j\in I$.
If $\cX_i-\cZ_i$ and $\cX_{ij}-\cZ_{ij}$ exist for all $i,j\in I$, then $\cX-\cZ$ exists.
\end{lemma}
\begin{proof}
By Lemma \ref{comp.2}, we can glue $\{\cX_i-\cZ_i\}_{i\in I}$ using Construction \ref{equiv.13}, and let $\cV$ be the resulting divided log space over $\cX$.
For every $\cY\in \lSpc/\cX$, $\cZ\times_{\cX} \cY=\emptyset$ if and only if $\cZ_i\times_{\cX_i}\cY=\emptyset$ for all $i\in I$.
Together with the isomorphism
\[
\Hom_{\cX}(\cY,\cV)
\simeq
\mathrm{Eq}(
\Hom_{\cX_i}(\cY\times_{\cX} \cX_i,\cV\times_{\cX}\cX_i)\rightrightarrows
\Hom_{\cX_{ij}}(\cY\times_{\cX} \cX_{ij},\cV\times_{\cX}\cX_{ij})),
\]
we deduce $\Hom_{\cX}(\cY,\cV)\simeq *$ whenever $\cZ\times_{\cX}\cY=\emptyset$.
\end{proof}

\begin{lemma}
\label{comp.4}
Let $Z\to X$ be a strict closed immersion in $\lFan/B$.
Then $h_X-h_Z$ exists, and there is an isomorphism
\[
h_X-h_Z
\simeq
h_{X-Z}.
\]
\end{lemma}
\begin{proof}
Suppose $\cY\in \lSpc/h_X$ and $h_Z\times_{h_X}\cY=\emptyset$.
Then there is a dividing Zariski cover $h_Y\to \cY$ such that the composition $h_Y\to h_X$ is equal to $h_f$ for some morphism $f$ in $\lFan/B$.
We have $Z\times_X Y=\emptyset$.
Hence there exists a unique morphism $u\colon Y\to X-Z$ over $X$.

Suppose that $v\colon h_Y\to h_{X-Z}$ is a morphism over $h_X$.
Then there exists a dividing cover $p\colon Y'\to Y$ such that the composite morphism $h_{Y'}\xrightarrow{vh_p} h_{X-Z}$ is equal to $h_w$ for some morphism $w\colon Y'\to X-Z$ over $X$.
By the universal property of open complements, we have $w=up$.
Hence we have $v=h_u$, i.e., $\Hom_{h_X}(h_Y,h_{X-Z})\simeq *$.

Proposition \ref{equiv.20} shows that $h_Y\times_{\cY} h_Y$ is representable.
Using this, we can similarly show $\Hom_{h_X}(h_Y\times_{\cY} h_Y,h_{X-Z})\simeq *$.
Together with the isomorphism
\[
\Hom_{h_X}(\cY,h_{X-Z})
\simeq
\mathrm{Eq}(\Hom_{h_X}(h_Y,h_{X-Z})\rightrightarrows \Hom_{h_X}(h_Y\times_{\cY} h_Y,h_{X-Z})),
\]
we obtain $\Hom_{h_X}(\cY,h_{X-Z})\simeq *$.
To conclude, observe $h_Z\times_{h_X}h_{X-Z}=\emptyset$.
\end{proof}

\begin{lemma}
\label{comp.6}
Let $h_Z\to h_X$ be a closed immersion in $\lSpc/B$, where $X,Z\in \lFan/B$.
Then $h_X-h_Z$ exists.
\end{lemma}
\begin{proof}
There exists a commutative square
\[
\begin{tikzcd}
h_{Z'}\ar[r,"f'"]\ar[d,"\simeq"']&
h_{X'}\ar[d,"\simeq"]
\\
h_Z\ar[r]&
h_X
\end{tikzcd}
\]
with vertical isomorphisms such that $f'$ is a strict closed immersion in $\lFan/B$.
Apply Lemma \ref{comp.4} to $f'$ to conclude.
\end{proof}

\begin{theorem}
\label{comp.5}
Let $i\colon \cZ\to \cX$ be a closed immersion in $\lSpc/B$.
Then $\cX-\cZ$ exists.
Furthermore,
the induced morphism $\cX-\cZ\to \cX$ is an open immersion.
\end{theorem}
\begin{proof}
By Proposition \ref{equiv.19}, there exists a representable Zariski cover $\amalg_{i\in I}h_{X_i}\to \cX$ with finite $I$ and $X_i\in \lFan/B$ such that each $h_{X_i}\to \cX$ is a representable open immersion.
We set $\cX_i:=h_{X_i}$.
Proposition \ref{equiv.20} shows that $\cX_{ij}:=\cX_i\times_{\cX} \cX_j$ is representable for all $i,j\in I$.
By Lemma \ref{comp.6}, $\cX_i-\cZ\times_{\cX}\cX_i$ and $\cX_{ij}-\cZ\times_{\cX}\cX_{ij}$ exist for all $i,j\in I$.
Together with Lemma \ref{comp.3},
we deduce that $\cX-\cZ$ exists.

Apply Lemma \ref{equiv.17} to the induced cartesian square
\[
\begin{tikzcd}
\amalg_{i\in I}(\cX-\cZ)\times_{\cX}\cX_i\ar[d]\ar[r]&
\amalg_{i\in I} \cX_i\ar[d]
\\
\cX-\cZ\ar[r]&
\cX
\end{tikzcd}
\]
to show that $\cX-\cZ\to \cX$ is a representable open immersion, i.e., an open immersion.
\end{proof}

\section{Blow-ups along closed immersions}
\label{blow-up}

As in the previous section, we define blow-ups by a universal property.
We also show that blow-ups exist in the log smooth case and are compatible with log smooth pullbacks.

\begin{definition}
For $X\in \lFan/B$, a strict closed subscheme $Z$ of $X$ is called an \emph{effective log Cartier divisor on $X$} if $\ul{Z\times_X X'}$ is an effective Cartier divisor on $\ul{X'}$ for all log smooth morphism $X'\to X$.
\end{definition}

We refer to Definition \ref{blow.10} for the notion of the blow-up $\Blow_Z X$ for all strict closed immersion $Z\to X$ in $\lSch/B$.

\begin{lemma}
\label{blow.20}
Let $i\colon Z\to X$ be a strict closed immersion in $\lSm/S$, where $S$ is an fs log scheme.
If $\ul{Z}$ is an effective Cartier divisor on $\ul{X}$, then $Z$ is an effective log Cartier divisor on $X$.
\end{lemma}
\begin{proof}
The projection $\Blow_Z X\to X$ is an isomorphism.
Hence by Lemma \ref{blow.1}, the projection
\[
\Blow_{Z'} X'\to X'
\]
is an isomorphism for all log smooth morphism $X'\to X$ of fs log schemes, where $Z':=Z\times_X X'$.
It follows that $\ul{Z'}$ is an effective Cartier divisor on $\ul{X'}$.
\end{proof}

\begin{definition}
\label{blow.5}
Let $\cZ\to \cX$ be a closed immersion in $\lSpc/B$.
We say that $\cZ$ is an \emph{effective log Cartier divisor on $\cX$} if there exists a cartesian square
\begin{equation}
\label{blow.5.1}
\begin{tikzcd}
h_{Z}\ar[r,"h_i"]\ar[d]&
h_{X}\ar[d]
\\
\cZ \ar[r]&
\cX
\end{tikzcd}
\end{equation}
such that the vertical morphisms are representable Zariski covers and the morphism $i\colon Z\to X$ in $\lFan/B$ exhibits $Z$ as an effective log Cartier divisor on $X$.
\end{definition}

\begin{lemma}
\label{blow.19}
Let $\cZ$ be an effective log Cartier divisor on $\cX$.
Then for every log smooth morphism $\cX'\to \cX$ in $\lSpc/B$, $\cZ\times_{\cX} \cX'$ is an effective log Cartier divisor on $\cX'$.
\end{lemma}
\begin{proof}
We have a cartesian square of the form \eqref{blow.5.1}.
Proposition \ref{property.1} yields a commutative square
\[
\begin{tikzcd}
h_{V}\ar[r,"h_g"]\ar[d]&
h_{U}\ar[d]
\\
\cX'\times_{\cX}h_X\ar[r]&
h_X
\end{tikzcd}
\]
such that $g$ is a log smooth morphism in $\lFan/B$ and the vertical morphisms are representable Zariski covers.
By Lemma \ref{equiv.27}, we can replace $U$ with its suitable dividing cover and $V$ by the corresponding pullback to assume that $h_U\to h_X$ is equal to $h_u$ for some dividing Zariski cover.
We have a commutative square
\[
\begin{tikzcd}
h_{Z\times_X V}\ar[d]\ar[r,"h_{i'}"]&
h_{V}\ar[d]
\\
\cZ\times_{\cX} \cX'\ar[r]&
\cX'
\end{tikzcd}
\]
such that the vertical morphisms are representable Zariski covers, where $i'$ is the projection.
To conclude, observe that $Z\times_X V$ is an effective log Cartier divisor on $V$ since the composition $V\to X$ is log smooth.
\end{proof}

\begin{lemma}
\label{blow.26}
Let $\cZ\to \cX$ be a closed immersion in $\lSpc/B$.
If there exists a representable Zariski cover $\cX'\to \cX$ such that $\cZ\times_{\cX}\cX'$ is an effective log Cartier divisor on $\cX'$, then $\cZ$ is an effective log Cartier divisor on $\cX$.
\end{lemma}
\begin{proof}
There exists a cartesian square
\[
\begin{tikzcd}
h_{Z'}\ar[r,"h_{i'}"]\ar[d]&
h_{X'}\ar[d]
\\
\cZ\times_{\cX} \cX'\ar[r]&
\cX'
\end{tikzcd}
\]
such that the vertical morphisms are representable Zariski covers and $i'$ exhibits $Z'$ as an effective log Cartier divisor on $X'$.
Hence we obtain a cartesian square
\[
\begin{tikzcd}
h_{Z'}\ar[r,"h_{i'}"]\ar[d]&
h_{X'}\ar[d]
\\
\cZ\ar[r]&
\cX
\end{tikzcd}
\]
such that the vertical morphisms are representable Zariski covers.
\end{proof}

\begin{definition}
For a closed immersion $\cZ\to \cX$ in $\lSpc/B$, the \emph{blow-up of $\cX$ along $\cZ$}, denoted $\Blow_{\cZ}\cX$, is defined to be a final object (if exists) of the full subcategory of $\lSpc/\cX$ consisting of $\cY\to \cX$ such that $\cZ\times_{\cX} \cY$ is an effective log Cartier divisor on $\cY$.

Suppose that $\cX'\to \cX$ is a morphism in $\lSpc/B$ with $\cZ':=\cX'\times_{\cX} \cZ$.
Then $\Blow_{\cZ'}\cX'\times_{\cX'}\cZ'\simeq \Blow_{\cZ'}\cX'\times_{\cX} \cZ$ is an effective log Cartier divisor on $\Blow_{\cZ'}\cX'$, so there is a canonical morphism
\begin{equation}
\Blow_{\cZ'}\cX'\to \Blow_{\cZ}\cX
\end{equation}
whenever the two blow-ups exist.
\end{definition}

\begin{lemma}
\label{blow.17}
Let $\cZ\to \cX$ be a closed immersion in $\lSpc/B$, and let $\cX'\to \cX$ be a log smooth morphism in $\lSpc/B$.
We set $\cZ':=\cZ\times_{\cX} \cX'$.
If $\Blow_{\cZ}\cX$ exists, then $\Blow_{\cZ'}\cX'$ exists, and there is an isomorphism
\[
\Blow_{\cZ'}\cX'
\simeq
\Blow_{\cZ}\cX\times_{\cX}\cX'.
\]
\end{lemma}
\begin{proof}
Apply Lemma \ref{blow.19} to $\Blow_{\cZ}\cX\times_{\cX} \cX'\to \Blow_{\cZ} \cX$ to show that $\Blow_{\cZ}\cX\times_{\cX} \cZ'$ is an effective log Cartier divisor on $\Blow_{\cZ}\cX\times_{\cX}\cX'$.
For every $\cY\in \lSpc/\cX'$, there is an isomorphism $\cZ\times_{\cX} \cY\simeq \cZ'\times_{\cX'}\cY$.
Use these to show that $\Blow_{\cZ}\cX\times_{\cX}\cX'$ is a final object of the full subcategory of $\lSpc/\cX'$ consisting of $\cY\to \cX'$ such that $\cZ'\times_{\cX'}\cY$ is an effective log Cartier divisor on $\cY$.
\end{proof}

\begin{lemma}
\label{blow.16}
Let $\cZ\to \cX$ be a closed immersion in $\lSpc/B$, and let $\amalg_{i\in I} \cX_i \to \cX$ be a representable Zariski cover with finite $I$ such that each $\cX_i\to \cX$ is a representable open immersion.
We set $\cX_{ij}:=\cX_i\times_{\cX} \cX_{j}$,
$\cZ_i:=\cZ\times_{\cX} \cX_i$, and $\cZ_{ij}:=\cZ\times_{\cX} \cX_{ij}$ for all $i,j\in I$.
If $\Blow_{\cZ_i}\cX_i$ and $\Blow_{\cZ_{ij}}\cX_{ij}$ exist for all $i,j\in I$, then $\Blow_{\cZ}\cX$ exists.
\end{lemma}
\begin{proof}
By Lemma \ref{blow.17}, we can glue $\{\Blow_{\cZ_i}\cX_i\}$ using Construction \ref{equiv.13}, and let $\cV$ be the resulting divided log space over $\cX$.
For every $\cY\in \lSpc/\cX$, $\cZ\times_{\cX} \cY$ is an effective log Cartier divisor on $\cY$ if and only if $\cZ_i\times_{\cX}\cY$ is an effective log Cartier divisor on $\cX_i\times_{\cX}\cY$ for all $i\in I$ by Lemma \ref{blow.26}.
Together with an isomorphism
\[
\Hom_{\cX}(\cY,\cV)
\simeq
\mathrm{Eq}(
\Hom_{\cX_i}(\cY\times_{\cX} \cX_i,\cV\times_{\cX}\cX_i)\rightrightarrows
\Hom_{\cX_{ij}}(\cY\times_{\cX} \cX_{ij},\cV\times_{\cX}\cX_{ij})),
\]
we deduce $\Hom_{\cX}(\cY,\cV)=*$ whenever $\cZ\times_{\cX}\cY$ is an effective log Cartier divisor on $\cY$.
To conclude, observe that $\cZ\times_{\cX} \cV$ is an effective log Cartier divisor on $\cV$ by Lemma \ref{blow.26}.
\end{proof}

\begin{lemma}
\label{blow.22}
Let $i\colon Z\to X$ be a strict closed immersion in $\lFan/B$.
If $h_Z$ is an effective log Cartier divisor on $h_X$, then there exists a dividing Zariski cover $Y\to X$ such that $Z\times_X Y$ is an effective log Cartier divisor on $Y$.
\end{lemma}
\begin{proof}
There exists a cartesian square
\[
\begin{tikzcd}
h_{Z'}\ar[r,"i'"]\ar[d]&
h_{X'}\ar[d]
\\
h_Z\ar[r,"h_i"]&
h_X
\end{tikzcd}
\]
such that the vertical morphisms are representable Zariski covers and $i'$ exhibits $Z'$ as an effective log Cartier divisor on $X'$.
By Lemma \ref{equiv.24}, we can replace $X'$ with its suitable dividing cover and $Z'$ by the corresponding pullback to assume that $h_{X'}\to h_X$ is equal to $h_u$ for some morphism $u\colon X'\to X$ in $\lFan/B$.

There is an isomorphism $q\colon h_{Z'} \xrightarrow{\simeq} h_{Z\times_X X'}$.
Apply Lemma \ref{equiv.28} to $q$ to obtain a dividing cover $v\colon V\to Z'$ and a morphism $r\colon V\to Z\times_X X'$ in $\lFan/B$ such that $qh_v=h_{r}$.
By Lemma \ref{equiv.24}, we obtain a dividing cover $v'\colon V'\to V$ and a morphism $s\colon V'\to Z$ such that the composition
\[
h_{V'}\xrightarrow{h_{v'}}h_V\xrightarrow{h_v}h_{Z'}\to h_Z
\]
is equal to $h_s$.
The square
\[
\begin{tikzcd}
h_{V'}\ar[r,"h_{i'vv'}"]\ar[d,"h_s"']&
h_{X'}\ar[d,"h_u"]
\\
h_Z\ar[r,"h_i"]&
h_X
\end{tikzcd}
\]
commutes.
Use Proposition \ref{div.5} to see that the square
\[
\begin{tikzcd}
V'\ar[r,"i'vv'"]\ar[d,"s"']&
X'\ar[d,"u"]
\\
Z\ar[r,"i"]&
X
\end{tikzcd}
\]
commutes after replacing $V'$ by a suitable dividing cover.

By Proposition \ref{equiv.23}, there exists a dividing cover $X_1'\to X'$ (resp.\ $X_2'\to X'$) such that the pullback $V'\times_{X'} X_1'\to Z'\times_{X'} X_1'$ (resp.\ $V'\times_{X'}X_2'\to (Z\times_{X} X')\times_{X'}X_2'$) is an isomorphism.
Take $Y:=X_1'\times_X X_2'$.
Then $Z\times_X Y$ is an effective log Cartier divisor on $Y$ since the closed immersion $Z\times_X Y\to Y$ is a pullback of $i'\colon Z'\to X'$ along the log smooth morphism $Y\to X'$.
\end{proof}

\begin{lemma}
\label{blow.6}
Let $i\colon Z\to X$ be a strict closed immersion of log smooth fs log schemes in $\lFan/S$, where $S\in \lFan/B$.
Then $\Blow_{h_Z}h_X$ exists, and there is an isomorphism
\[
\Blow_{h_Z}h_X
\simeq
h_{\Blow_Z X}.
\]
\end{lemma}
\begin{proof}
Suppose that $\cY\in \lSpc/h_X$ and $h_Z\times_{h_X}\cY$ is an effective log Cartier divisor on $\cY$.
Choose a dividing Zariski cover $h_Y\to \cY$ with $Y\in \lFan/B$.
By Lemma \ref{equiv.24}, after replacing $Y$ by its suitable dividing cover, the composition $h_Y\to \cY\to h_X$ is equal to $h_f$ for some morphism $f\colon Y\to X$ in $\lFan/B$.
Since $h_Y\to \cY$ is log smooth, Lemma \ref{blow.19} shows that $h_{Z\times_X Y}$ is an effective log Cartier divisor on $h_Y$.
Hence by Lemma \ref{blow.22}, after replacing $Y$ by its suitable dividing Zariski cover, $Z\times_X Y$ is an effective log Cartier divisor on $Y$.

Since $\ul{Z\times_X Y}$ is an effective Cartier divisor on $\ul{Y}$, there exists a unique morphism $\ul{Y}\to \ul{\Blow_Z X}$ over $\ul{X}$ by the universal property of blow-ups.
Together with $\Blow_Z X\simeq \ul{\Blow_Z X}\times_{\ul{X}} X$, we deduce that there exists a unique morphism $u\colon Y\to \Blow_Z X$ over $X$.

Suppose that $v\colon h_Y\to h_{\Blow_Z X}$ is a morphism over $h_X$.
By Lemma \ref{equiv.24}, there exists a dividing cover $p\colon Y'\to Y$ such that the composite morphism $h_{Y'}\xrightarrow{vh_p} h_{\Blow_Z X}$ is equal to $h_w$ for some morphism $w\colon Y'\to \Blow_Z X$ in $\lFan/B$.
The universal property of blow-ups shows $w=up$.
Hence we have $v=h_u$, i.e., $\Hom_{h_X}(h_Y,h_{\Blow_Z X})\simeq *$.

By Proposition \ref{equiv.20}, $h_Y\times_{\cY} h_Y$ is representable.
Using this, we can similarly show $\Hom_{h_X}(h_Y\times_{\cY} h_Y,h_{\Blow_Z X})\simeq *$.
Together with the isomorphism
\[
\Hom_{h_X}(\cY,h_{\Blow_Z X})
\simeq
\mathrm{Eq}(\Hom_{h_X}(h_Y,h_{\Blow_Z X})\rightrightarrows \Hom_{h_X}(h_Y\times_{\cY} h_Y,h_{\Blow_Z X})),
\]
we obtain $\Hom_{h_X}(\cY,h_{\Blow_Z X})\simeq *$.
To conclude, observe that $\Blow_Z X\times_X Z$ is an effective log Cartier divisor on $\Blow_Z X$ by Lemmas \ref{blow.2} and \ref{blow.20}.
\end{proof}

\begin{lemma}
\label{blow.21}
Let $f\colon X\to S$ be a log smooth morphism in $\lFan/B$.
If $\cZ\to h_X$ is a closed immersion such that the composition $\cZ\to h_S$ is a log smooth morphism in $\lSpc/B$, then there exists a cartesian square
\[
\begin{tikzcd}
h_W\ar[r,"h_a"]\ar[d]&
h_V\ar[d,"h_v"]
\\
\cZ\ar[r]&
h_X
\end{tikzcd}
\]
such that $a$ is a strict closed immersion, $fva$ is log smooth, and $v$ is a dividing cover.
\end{lemma}
\begin{proof}
By Lemma \ref{equiv.26}, there exists a commutative square
\[
\begin{tikzcd}
h_{Z}\ar[r,"h_i"]\ar[d,"\simeq"']&
h_{X'}\ar[d,"h_u"',"\simeq"]
\\
\cZ\ar[r]&
h_X
\end{tikzcd}
\]
with vertical isomorphisms such that $i$ is a strict closed immersion and $u$ is a dividing cover.
Apply Lemma \ref{equiv.27} and Proposition \ref{property.4} to the log smooth morphism $h_Z\to h_S$ to obtain a dividing cover $z\colon Z'\to Z$ such that the composition $h_{Z'}\to h_S$ is equal to $h_g$ for some log smooth morphism $g\colon Z'\to S$ in $\lFan/B$.
Use Proposition \ref{div.5} to obtain a dividing cover $z'\colon Z''\to Z'$ such that the two compositions
\[
h_{Z''}\xrightarrow{h_{z'}} h_{Z'}\xrightarrow{h_z}h_Z\xrightarrow{h_{iuf}}h_S
\text{ and }
h_{Z''}\xrightarrow{h_{z'}}h_{Z'}\xrightarrow{h_g}h_S
\]
are equal.
The composition $Z''\to X'$ is a monomorphism, so Proposition \ref{equiv.23} yields a dividing cover $X''\to X'$ such that the projection $Z''\times_{X'} X''\to X''$ is strict.
Since the projection $Z\times_{X'}X''\to X''$ is strict, the induced morphism $Z''\times_{X'}X''\to Z\times_{X'}X''$ is a strict dividing cover, i.e., an isomorphism by Proposition \ref{equiv.11}(3).
The composition $Z''\times_{X'}X''\to S$ is log smooth, so the composition $Z\times_{X'}X''\to S$ is log smooth too.
Take $V:=X''$ and $W:=Z\times_{X'}X''$ to conclude.
\end{proof}

For $\cS\in \lSpc/B$,
let $\lSmSpc/\cS$ denote the full subcategory of $\lSpc/\cS$ consisting of log smooth morphisms $\cX\to \cS$.
\begin{theorem}
\label{blow.7}
Let $\cZ\to \cX$ be a closed immersion in $\lSmSpc/\cS$, where $\cS\in \lSpc/B$.
Then $\Blow_{\cZ}\cX$ exists, and $\Blow_{\cZ}\cX\in \lSmSpc/\cS$.
\end{theorem}
\begin{proof}
By Proposition \ref{property.1}, there exists a commutative square
\[
\begin{tikzcd}
h_X\ar[d]\ar[r,"h_f"]&
h_S\ar[d]
\\
\cX\ar[r]&
\cS
\end{tikzcd}
\]
such that $f$ is a log smooth morphism in $\lFan/B$ and the vertical morphisms are representable Zariski covers.
After replacing $X$ by its suitable dividing cover, by Lemma \ref{blow.21}, we may assume that there exists a cartesian square
\[
\begin{tikzcd}
h_Z\ar[r,"h_i"]\ar[d]&
h_X\ar[d]
\\
\cZ\ar[r]&
\cX
\end{tikzcd}
\]
such that $i$ is a strict closed immersion in $\lFan/B$ and $fi$ is log smooth.

By Proposition \ref{equiv.19}, we may further assume that $X\simeq \amalg_{i\in I} X_i$ with finite $I$ such that each $h_{X_i}\to \cX$ is a representable open immersion.
For all $i,j\in I$, there exists a commutative square
\[
\begin{tikzcd}
h_{U_{ij}}\ar[r,"u_{ij}"]\ar[d,"\simeq"']&
h_{X_i}\ar[d,"\simeq"]
\\
\cX_i\times_{\cX}\cX_j\ar[r]&
\cX_i
\end{tikzcd}
\]
with vertical isomorphisms such that $u_{ij}$ is a representable open immersion.

Since $X_i,Z\times_X X_i,U_{ij},Z\times_X U_{ij}\in \lSm/S$, Lemma \ref{blow.6} shows that $\Blow_{\cZ_i}\cX_i$ and $\Blow_{h_{Z_i}}h_{X_{ij}}$ exist.
Lemma \ref{blow.16} finishes the proof.
\end{proof}

Let $\boxx$ be the fs log scheme whose underlying scheme is $\P^1$ and whose log structure is the compactifying log structure associated with the open immersion $\A^1\to \P^1$ away from $\infty$.

\begin{definition}
\label{blow.15}
Let $i\colon \cZ\to \cX$ be a closed immersion in $\lSmSpc/\cS$, where $\cS\in \lSpc/B$.
The \emph{deformation space associated with $i$} is defined to be
\[
\Deform_{\cZ}\cX
:=
\Blow_{\cZ\times \{0\}}(\cX\times \boxx)-\Blow_{\cZ\times \{0\}}(\cX\times \{0\}).
\]
The \emph{normal bundle of $\cZ$ in $\cX$} is defined to be
\[
\Normal_{\cZ}\cX
:=
\Deform_{\cZ}\cX \times_{\boxx}\{0\}.
\]
These exist by Theorems \ref{comp.5} and \ref{blow.7}.
\end{definition}

Suppose that $Z\to X$ is a strict closed immersion in $\sSm/S$ with $S\in \lSch/B$, where $\sSm$ denotes the class of strict smooth morphisms.
Let $\Normal_Z X:=\Normal_{\ul{Z}}\ul{X}\times_{\ul{X}}X$ denote the normal bundle of $Z$ in $X$.
As explained in \cite[Definition 7.5.1]{logDM}, there is an isomorphism
\begin{equation}
\Deform_Z X\times_{\boxx}\{0\}
\simeq
\Normal_Z X,
\end{equation}
where $\Deform_Z X:=\Blow_{Z\times \{0\}}(X\times \boxx)-\Blow_{Z\times \{0\}}(X\times \{0\})$.
Hence we have an isomorphism
\begin{equation}
\label{blow.15.1}
\Normal_{h_Z}h_X
\simeq
h_{\Normal_Z X}
\end{equation}
by Lemmas \ref{comp.4} and \ref{blow.6} since showing \eqref{blow.15.1} is Zariski local on $X$ and $Z$.

\begin{remark}
For a closed immersion of schemes $Z\to X$, Verdier \cite{zbMATH03522129} used $\A^1$ to define a deformation space, while Fulton \cite{zbMATH01027930} used $\P^1$.
We use $\boxx$ since this choice is suitable for log motivic homotopy theory, see e.g.\ \cite[Theorem 7.5.4]{logDM}.
\end{remark}

\begin{proposition}
\label{blow.28}
Let $i\colon Z\to X$ be a closed immersion in $\lSm/S$ with $S\in \lSch/B$.
Then the induced morphism
\begin{equation}
\label{blow.28.1}
i^*\Omega_{X/S}^1\to \Omega_{Z/S}^1
\end{equation}
is surjective, and its kernel is a locally free $\cO_Z$-module.
Furthermore, if $V$ is the vector bundle over $Z$ associated with the dual of the kernel of \eqref{blow.28.1}, then there is an isomorphism
\begin{equation}
\label{blow.28.2}
\Normal_{h_Z}h_X
\simeq
h_V.
\end{equation}
\end{proposition}
\begin{proof}
The question is strict \'etale local on $X$ and $Z$.
Hence by \cite[Proposition III.2.3.5]{Ogu}, we may assume that $i$ admits a factorization $Z\xrightarrow{i'}X'\xrightarrow{u} X$ such that $i'$ is a strict closed immersion and $u$ is a log \'etale monomorphism.
Then \eqref{blow.28.1} is isomorphic to the induced morphism
\[
i'^*\Omega_{X'/S}^1 \to \Omega_{Z/S}^1.
\]
Together with \cite[Theorem IV.3.2.2]{Ogu}, we see that \eqref{blow.28.1} is surjective and its kernel is a locally free $\cO_Z$-module.

To show \eqref{blow.28.2}, it suffices to show
\[
\Normal_{h_Z}h_{X'}
\simeq
h_V
\]
since $h_u\colon h_{X'}\to h_X$ is an open immersion.
Hence we can replace $Z\xrightarrow{i} X$ by $Z\xrightarrow{i'} X'$, so we may assume that $i$ is a strict closed immersion.

Recall that the question is strict \'etale local on $X$.
By Proposition \ref{blow.3}, we may assume that there exists a strict smooth morphism $X\to Y$ in $\lSm/S$ such that the composition $Z\to Y$ is strict smooth too.
We finish the proof by applying \eqref{blow.15.1} to the strict closed immersion $Z\to X$ in $\sSm/S$.
\end{proof}

\begin{lemma}
\label{blow.24}
Let $W\to Z\to X$ be closed immersions of schemes.
Then there is a cartesian square
\begin{equation}
\label{blow.24.1}
\begin{tikzcd}
\Deform_W Z\ar[d]\ar[r]&
\Deform_W X\ar[d]
\\
Z\times \boxx\ar[r]&
\Deform_Z X.
\end{tikzcd}
\end{equation}
\end{lemma}
\begin{proof}
The question is Zariski local on $X$, so we reduce to the case when $X=\Spec(A)$, $Z=\Spec(A/I)$, and $W=\Spec(A/J)$, where $I\subset J$ are ideals of $A$.
According to \cite[Section 5.1]{zbMATH01027930}, we have explicit descriptions
\begin{gather*}
\Blow_{Z\times \{0\}}(X\times \A^1)-\Blow_{Z\times \{0\}} (X\times \{0\})
\simeq
\Spec\big(\bigoplus_{n\in \Z} I^{-n} t^n\big),
\\
\Blow_{W\times \{0\}}(X\times \A^1)-\Blow_{W\times \{0\}}(X\times \{0\})
\simeq
\Spec\big(\bigoplus_{n\in \Z} J^{-n} t^n\big),
\end{gather*}
where $t$ is an indeterminate and $I^{n},J^{n}:=A$ for all integer $n\leq 0$.
The closed subscheme $Z\times \A^1$ of $\Blow_{Z\times \{0\}}(X\times \A^1)-\Blow_{Z\times \{0\}} (X\times \{0\})$ is given by the ideal generated by $It^{-1}$.
Hence we see that
\[
(Z\times \A^1) \times_{\Blow_{Z\times \{0\}}(X\times \A^1)-\Blow_{Z\times \{0\}} (X\times \{0\})}(\Blow_{W\times \{0\}}(X\times \A^1)-\Blow_{W\times \{0\}} (X\times \{0\}))
\]
is isomorphic to 
\[
\Blow_{W\times \{0\}}(Z\times \A^1)-\Blow_{W\times \{0\}} (Z\times \{0\})\simeq \Spec\big(\bigoplus_{n\in \Z} (J/I)^{-n} t^n\big),
\]
where $(J/I)^n:=A/I$ for all integer $n\leq 0$.
It follows that \eqref{blow.24.1} is cartesian.
\end{proof}

\begin{proposition}
\label{blow.23}
Let $\cW\to \cZ\to \cX$ be closed immersions in $\lSmSpc/\cS$, where $\cS\in \lSpc/B$.
Then there is a cartesian square
\begin{equation}
\label{blow.23.2}
\begin{tikzcd}
\Deform_{\cW}\cZ\ar[d]\ar[r]&
\Deform_{\cW}\cX\ar[d]
\\
\cZ\times \boxx\ar[r]&
\Deform_{\cZ}\cX.
\end{tikzcd}
\end{equation}
\end{proposition}
\begin{proof}
As in the proof of Theorem \ref{blow.7}, there exists a commutative diagram
\[
\begin{tikzcd}
h_Z\ar[r,"h_i"]\ar[d]&
h_X\ar[r,"h_f"]\ar[d]&
h_S\ar[d]
\\
\cZ\ar[r]&
\cX\ar[r]&
\cS
\end{tikzcd}
\]
such that $i$ is a strict closed immersion, $f$ and $fi$ are log smooth, the vertical morphisms are representable Zariski covers, and the left small square is cartesian.
By Lemma \ref{blow.21}, there exists a cartesian square
\[
\begin{tikzcd}
h_{W'}\ar[d,"\simeq"']\ar[r,"h_{a}"]&
h_{Z'}\ar[d,"h_u"',"\simeq"]
\\
\cW\times_{\cZ}h_Z\ar[r]&
h_Z
\end{tikzcd}
\]
with vertical isomorphisms such that $u$ is a dividing cover, $a$ is a strict closed immersion, and the composition $W'\to S$ is log smooth.
The composition $Z'\to X$ is a proper monomorphism, so Proposition \ref{equiv.23} yields a dividing cover $X''\to X$ such that the projection $Z'':=Z'\times_X X''\to X''$ is a strict closed immersion.
Hence we obtain a commutative diagram
\[
\begin{tikzcd}
h_{W''}\ar[d]\ar[r,"h_{a''}"]&
h_{Z''}\ar[d]\ar[r,"h_{i''}"]&
h_{X''}\ar[d]
\\
\cW\ar[r]&
\cZ\ar[r]&
\cX
\end{tikzcd}
\]
with cartesian squares such that $W'':=W\times_X X''$, the vertical morphisms are representable Zariski covers, $i''$ and $a''$ are strict closed immersions, and the compositions $X'',Z'',W''\to S$ are log smooth.

The proof is done if we prove the following steps:
\begin{enumerate}
\item[(1)] Show that $(\cZ\times \{0\})\times_{\cX\times \boxx}\Deform_{\cW}\cX$ is an effective log Cartier divisor on $\Deform_{\cW}\cX$.
\item[(2)] Show $\Blow_{\cZ\times \{0\}}(\cX\times \{0\})\times_{\Blow_{\cZ\times \{0\}}(\cX\times \boxx),p} \Deform_{\cW}\cX=0$, where the morphism $p$ is obtained by (1).
\item[(3)] Show that \eqref{blow.23.2} is cartesian, where its right vertical morphism is obtained by (2).
\end{enumerate}
The steps (1)--(3) are Zariski local on $\cX$ by Lemmas \ref{comp.2}, \ref{blow.19}, \ref{blow.26}, and \ref{blow.17}.
Hence we reduce to showing the similar steps for $h_{W''}\to h_{Z''}\to h_{X''}\to h_{S''}$.
Lemma \ref{blow.24} proves the steps (1)--(3) at once.
\end{proof}

\begin{corollary}
\label{blow.25}
Let $\cW\to \cZ\to \cX$ be closed immersions in $\lSmSpc/\cS$, where $\cS\in \lSpc/B$.
Then there is a cartesian square
\begin{equation}
\label{blow.25.1}
\begin{tikzcd}
\Normal_{\cW}\cZ\ar[d]\ar[r]&
\Normal_{\cW}\cX\ar[d]
\\
\cZ\ar[r]&
\Normal_{\cZ}\cX.
\end{tikzcd}
\end{equation}
\end{corollary}
\begin{proof}
The square \eqref{blow.25.1} is obtained by a pullback of \eqref{blow.23.2}.
\end{proof}

Normal bundles of divided log spaces can be regarded as affine bundles in the following sense.

\begin{proposition}
\label{blow.27}
Let $\cZ\to \cX$ be a closed immersion in $\lSmSpc/\cS$, where $\cS\in \lSpc/B$.
Then there exists a cartesian square
\[
\begin{tikzcd}
h_{Z\times \A^n}\ar[d]\ar[r,"h_p"]&
h_Z\ar[d]
\\
\Normal_{\cZ} \cX\ar[r]&
\cZ
\end{tikzcd}
\]
with $Z\in \lFan/B$ and $n\in \N$ such that $h_Z\to \cZ$ is a representable Zariski cover and $p$ is the projection.
\end{proposition}
\begin{proof}
As in the proof of Theorem \ref{blow.7}, there exists a commutative diagram
\[
\begin{tikzcd}
h_Z\ar[d]\ar[r,"h_i"]&
h_X\ar[d]\ar[r,"h_f"]&
h_S\ar[d]
\\
\cZ\ar[r]&
\cX\ar[r]&
\cS
\end{tikzcd}
\]
with $S,X,Z\in \lFan/B$ such that the left square is cartesian, the vertical morphisms are representable Zariski covers, $f$ and $fi$ are log smooth, and $i$ is a strict closed immersion.
Lemmas \ref{blow.17} and \ref{blow.6} yield isomorphisms
\[
\Normal_{\cZ} \cX\times_{\cX}h_X
\simeq
\Normal_{h_Z}h_X
\simeq
h_{\Normal_Z X}.
\]
Together with $h_Z\simeq \cZ\times_{\cX} h_X$, we obtain an isomorphism
\[
\Normal_{\cZ} \cX\times_{\cZ} h_Z
\simeq
h_{\Normal_Z X}.
\]
Since $\ul{\Normal_Z X}$ is a vector bundle over $\ul{Z}$, we obtain a desired cartesian square after further Zariski localization on $Z$.
\end{proof}

\appendix

\section{Charts for log smooth morphisms}

The chart theorem for log smooth morphisms \cite[Theorem IV.3.3.1]{logDM} is crucial for the development of the theory of log motives since this allows us to understand the structure of log smooth morphisms more concretely.
However, the theorem is strict \'etale local on the source even though we are working with Zariski log structures.
In this section, we explain how the theorem can be Zariski local on the source with a stronger assumption.

\begin{definition}
Let $X$ be an fs log scheme, and let $x$ be a point.
We say that a chart $P$ of $X$ is called \emph{neat at $x$} if $P$ is sharp and $P\to \ol{\cM}_{X,x}$ is an isomorphism.
See \cite[Definition II.2.3.1]{Ogu} for other equivalent conditions.
\end{definition}

\begin{definition}
Let $f\colon X\to S$ be a morphism of fs log schemes.
We set
\[
\cM_{X/S}:=\coker(\cM_{\ul{X}\times_{\ul{S}}S} \to \cM_X),
\]
where the cokernel is taken in the category of sheaves of monoids on $X$.
By \cite[Proposition I.1.3.3]{Ogu}, there is an isomorphism
\[
\cM_{X/S}^\gp
\simeq
\coker(\cM_{\ul{X}\times_{\ul{S}}S}^\gp \to \cM_X^\gp).
\]
Let $x$ be a point of $X$.
A chart $\theta\colon P\to Q$ for $f$ is called \emph{neat at $x$} if the induced sequence
\[
0\to P^\gp \to Q^\gp \to \cM_{X/S,x}^\gp \to 0
\]
is exact.
This is a rephrase of the conditions in \cite[Theorem II.2.4.4]{Ogu}.
\end{definition}

\begin{proposition}
Let $f\colon X\to S$ be an exact morphism of fs log schemes.
If $\theta\colon P\to Q$ is a chart for $f$ neat at $x\in X$ such that $P$ is a neat chart at $f(x)$, then $Q$ is a neat chart at $x$.
\end{proposition}
\begin{proof}
This is proved in \cite[Remark II.2.4.5]{Ogu} with the assumption that the induced homomorphism $\ol{\cM}_{S,f(x)}\to \ol{\cM}_{X,x}$ is injective.
The exactness of $f$ implies this assumption by \cite[Proposition I.4.2.1(5)]{Ogu}.
\end{proof}

\begin{proposition}
\label{logsmooth.1}
Let $f\colon X\to S$ be a log smooth (resp.\ log \'etale) morphism of fs log schemes, and let $x$ be a point of $X$.
Assume that $\cM_{X/S,x}^\gp$ is torsion free and $S$ has a chart $P$.
Then in a Zariski neighborhood of $x$, $f$ admits a chart $\theta\colon P\to Q$ neat at $x$, and the induced morphism $X\to S\times_{\A_P} \A_Q$ is strict smooth (resp.\ strict \'etale).
\end{proposition}
\begin{proof}
By \cite[Theorem III.1.2.7(4)]{Ogu}, $f$ admits a chart including $P$ neat at $x$ in a Zariski neighborhood of $x$.
Then argue as in the proof of \cite[Theorem IV.3.3.1]{Ogu} to conclude.
\end{proof}

\section{Exact monomorphisms}

\begin{definition}
For an fs monoid $P$ and a ring $R$, we set $\A_{P,R}:=\Spec(P\to R[P])$, see \cite[Definition III.1.2.3]{Ogu}.
If $I$ is an ideal of $P$, we set
\[
\A_{(P,I),R}
:=
\A_{P,R}\times_{\Spec(R[P])}\Spec(R[P]/I).
\]
If $P$ is sharp, we set
\[
\pt_{P,R}
:=
\A_{(P,P^+),R},
\]
where $P^+$ denotes the ideal of non-units of $P$.
We often omit $R$ in this notation when it is clear from the context.
\end{definition}

\begin{lemma}
\label{equiv.10}
Let $\theta\colon P\to Q$ be an exact local homomorphism of sharp fs monoids, and let $k$ be a field.
If the induced morphism $f\colon \pt_{Q,k}\to \pt_{P,k}$ is a monomorphism of fs log schemes, then $\theta$ is an isomorphism.
\end{lemma}
\begin{proof}
Let us omit $k$ for simplicity of notation.
Note that $\theta$ is injective by \cite[Proposition I.4.2.1(5)]{Ogu}.
We set $Z:=\pt_Q\times_{\pt_P}\pt_Q$.
If $q\in Q$ satisfies $(q,0)\in (Q\oplus_P Q)^*$,
then there exists an integer $n>0$ such that $(nq,0)+(q_1,q_2)=0$ for some $q_1,q_2\in Q$.
There exists $p\in P^\gp$ such that $nq+q_1=p$ and $q_2=-p$ in $Q^\gp$.
Since $\theta$ is exact, we have $p=nq+q_1\in P^\gp\cap Q=P$.
Similarly,
we have $p\in -P$ and hence $p\in P^*=0$.
This implies $q\in Q^*=0$.
Hence the first inclusion $Q\to Q\oplus_P Q$ sends $Q^+$ into $(Q\oplus_P Q)^+$, and the same holds for the second inclusion too,
so $Z$ contains $\A_{(Q\oplus_P Q,(Q\oplus_P Q)^+)}$ as a strict closed subscheme.
It follows that $Z$ contains a point $z$ such that $\ol{\cM}_{Z,z}\simeq \ol{Q\oplus_P Q}$.
Since $f$ is a monomorphism, the diagonal morphism $\pt_Q\to Z$ is an isomorphism.
This gives an isomorphism
\[
\overline{Q\oplus_P Q}
\simeq
Q.
\]

Due to \cite[Proposition I.4.2.5(5)]{Ogu}, we have an equality
\[
2\rank(Q^{\gp})-\rank(P^{\gp})
=
\rank((\overline{Q\oplus_P Q})^{\gp}).
\]
We deduce that $P^{\gp}$ and $Q^{\gp}$ have the same rank.
By \cite[Proposition I.4.2.1(5)]{Ogu}, $\theta$ is injective.
Together with \cite[Proposition I.4.3.5]{Ogu}, we see that $\theta$ is Kummer.

The underlying scheme of $\A_{(Q\oplus_P Q,(Q\oplus_P Q)^+)}$ is $\A_{(Q\oplus_P Q)^*}$, which is isomorphic to $\A_{Q^\gp/P^\gp}$ by \cite[Lemma 3.3]{MR1922832} using the fact that $\theta$ is Kummer.
Since the morphism $\pt_Q\to Z$ is an isomorphism,
$Q^\gp/P^\gp$ is $0$.
Hence $\theta$ is an isomorphism.
\end{proof}

\begin{proposition}
\label{equiv.32}
Let $f\colon Y\to X$ be an exact monomorphism of fs log schemes.
Then $f$ is strict.
\end{proposition}
\begin{proof}
Let $y$ be a point of $Y$.
We set $x:=f(y)$, $P:=\ol{\cM}_{X,x}$, and $Q:=\ol{\cM}_{Y,y}$.
Let $\theta\colon P\to Q$ be the induced homomorphism.
The restriction of $f$ at $x$ and $y$ is a morphism $g\colon \pt_{Q,k'}\to \pt_{P,k}$ for some fields $k$ and $k'$.
The morphism $g$ is an exact monomorphism too.

Consider the canonical factorization
\[
\pt_{Q,k'}
\xrightarrow{g'}
\pt_{P,k'}
\to
\pt_{P,k}.
\]
Since $g$ is an exact monomorphism, $g'$ is an exact monomorphism.
By Lemma \ref{equiv.10}, $\theta$ is an isomorphism.
Hence $f$ is strict at $x$.
\end{proof}

\section{Strict closed immersions of log smooth schemes}

\begin{proposition}
\label{blow.3}
Let $i\colon Z\to X$ be a strict closed immersion in $\lSm/S$, where $S$ is an fs log scheme.
Then strict \'etale locally on $X$, there exists a cartesian square
\begin{equation}
\label{blow.3.1}
\begin{tikzcd}
Z\ar[d,"i"']\ar[r]&
Y\ar[d,"i_0"]
\\
X\ar[r,"u"]&
Y\times \A^s
\end{tikzcd}
\end{equation}
with $Y\in \lSm/S$ such that $i_0$ is the zero section and $u$ is strict \'etale,
where $\A^s$ has trivial log structure.
If $\cM_{X/S}^\gp$ is torsion free, then \eqref{blow.3.1} exists Zariski locally on $X$.
\end{proposition}
\begin{proof}
Let $x$ be a point of $Z$, and let $\cI$ be the sheaf of ideals on $X$ defining $Z$.
By \cite[Lemma IV.1.2.10, Theorem IV.3.2.2]{Ogu}, we can choose local sections $m_1,\ldots,m_r$ of $\cM_X$ and $m_{r+1},\ldots,m_{r+s}$ of $\cI$ such that $\{dm_1,\ldots,dm_{r+s}\}$ (resp.\ $\{dm_{1},\ldots,dm_{r}\}$) gives rise a basis of $\Omega_{X/S,x}^1$ (resp.\ $\Omega_{Z/S,x}^1$).
Zariski locally on $X$, the local sections $m_1,\ldots,m_{r+s}$ are global sections.
Hence Zariski locally on $X$, we obtain a cartesian square
\[
\begin{tikzcd}
Z\ar[d,"i"']\ar[r]&
S\times \A_{\N^r}\ar[d]
\\
X\ar[r,"v"]&
S\times \A_{\N^r}\times \A^s
\end{tikzcd}
\]
such that the right vertical morphism is the zero section.
According to the proof of \cite[Theorem IV.3.2.6]{Ogu}, $v$ is log \'etale.

We may assume that $S\times \A_{\N^r}$ admits a chart $P$.
By \cite[Theorem IV.3.3.1]{Ogu}, strict \'etale locally on $X$, $v$ admits a chart $\theta\colon P\to Q$ such that the induced morphism $X\to (S\times \A_{\N^r})\times_{\A_P}\A_Q \times \A^s$ is strict \'etale.
By setting $Y:=(S\times \A_{\N^r})\times_{\A_P}\A_Q$, we obtain \eqref{blow.3.1}.

If $\cM_{Y/X}^\gp$ is torsion free, use Proposition \ref{logsmooth.1} instead.
\end{proof}

\section{Blow-ups along strict closed subschemes}

\begin{definition}
\label{blow.10}
Suppose that $i\colon Z\to X$ is a strict closed immersion in $\lSch/S$, where $S$ is an fs log scheme.
The \emph{blow-up of $X$ along $Z$} is defined to be
\[
\Blow_Z X
:=
\Blow_{\ul{Z}}\ul{X}\times_{\ul{X}}X,
\]
where $\Blow_{\ul{Z}}\ul{X}$ denotes the usual blow-up.
\end{definition}

\begin{lemma}
\label{blow.2}
Let $i\colon Z\to X$ be a strict closed immersion in $\lSm/S$, where $S$ is an fs log scheme.
Then $\ul{\Blow_Z X\times_X Z}$ is an effective Cartier divisor on $\ul{\Blow_Z X}$ and we have $\Blow_Z X,\Blow_Z X\times_X Z\in \lSm/S$.
\end{lemma}
\begin{proof}
The question is strict \'etale local on $X$ by \cite[Theorem 0.2]{zbMATH06164842} and Remark \ref{blow.8} below, so we may assume the existence of the diagram \eqref{blow.3.1}.
Since the morphism $u$ in this diagram is strict flat, there is a canonical isomorphism
\begin{equation}
\label{blow.2.1}
\Blow_Z X
\simeq
\Blow_Y(Y\times \A^s) \times_{Y\times \A^s}X.
\end{equation}
To conclude, observe that the claim for the strict closed immersion $Y\to Y\times \A^s$ is clear.
\end{proof}

\begin{remark}
\label{blow.8}
We are working with fs log schemes with Zariski log structures.
On the other hand,
the log smoothness and log \'etaleness in \cite[Theorem 0.2]{zbMATH06164842} are about fs log schemes with \'etale log structures.
The category of fs log schemes with Zariski log structures is equivalent to the category of fs log schemes with \'etale log structures having charts Zariski locally by \cite[Proposition III.1.4.1]{Ogu}.
To use \cite[Theorem 0.2]{zbMATH06164842} for the setting of Zariski log structures,
we need Proposition \ref{blow.11} below.
We thank one of the referees for pointing out this issue.
\end{remark}

\begin{proposition}
\label{blow.11}
Let $f\colon Y\to X$ be a morphism of fs log schemes with Zariski log structures.
Then it is log smooth (resp.\ log \'etale) if and only if its associated morphism of fs log schemes with \'etale log structures is log smooth (resp.\ log \'etale).
\end{proposition}
\begin{proof}
The if direction is an immediate consequence of \cite[Proposition IV.3.1.4(2)]{Ogu}.
For the only if direction,
let $S\to Y$ be an open immersion of fs log schemes with \'etale log structures.
Observe that $S$ has charts Zariski locally.
By \cite[Corollary IV.3.1.5]{Ogu},
it suffices to show that for every first-order log thickening of fs log schemes with \'etale log structures $S\to T$ in the sense of \cite[Definition IV.2.1.1]{Ogu},
$T$ has charts Zariski locally.
This is a consequence of \cite[Proposition II.2.3.7]{Ogu} and Lemma \ref{blow.9} below.
\end{proof}

\begin{lemma}
\label{blow.9}
Let $i\colon S\to T$ be a first-order log thickening of fs log schemes with \'etale log structures.
Assume that $\ul{S}$ is affine.
If $g\colon S\to \A_P$ is a strict morphism of fs log schemes with \'etale log structures for some fs monoid $P$ such that $P^\gp$ is torsion free (e.g., a sharp fs monoid),
then it lifts to a strict morphism of fs log schemes with \'etale log structures $h\colon T\to \A_P$.
\end{lemma}
\begin{proof}
Since $\A_P\to \Spec(\Z)$ is log smooth by \cite[Corollary IV.3.1.9]{Ogu},
we have a lift $h$ by \cite[Proposition IV.3.1.4(2)]{Ogu}.
We need to show that $h$ is strict.
By \cite[Th\'eor\`eme IV.18.1.2]{EGA},
the small \'etale sites $\ul{S}_{\et}$ and $\ul{T}_{\et}$ are equivalent.
Hence it suffices to show that the homomorphism $\ol{\cM}_{\A_P,g(\ol{s})} \to \ol{\cM}_{T,i(\ol{s})}$ induced by $h$ is an isomorphism for every geometric point $\ol{s}$ of $\ul{S}$.
This is a consequence of the assumption that $g$ and $i$ are strict.
\end{proof}

\begin{lemma}
\label{blow.1}
Let $i\colon Z\to X$ be a strict closed immersion in $\lSm/S$, where $S$ is an fs log scheme.
Then for every log smooth morphism $X'\to X$, there is a canonical isomorphism
\[
\Blow_{Z'} X'\simeq\Blow_Z X\times_X X',
\]
where $Z':=Z\times_X X'$.
\end{lemma}
\begin{proof}
The question is strict \'etale local on $X$ and $X'$, so we may assume that \eqref{blow.3.1} exists and $X$ admits a chart $P$.
By \cite[Theorem IV.3.3.1]{Ogu}, we may also assume that there exists a chart $P\to Q$ of $X'\to X$ such that the induced morphism
\[
X'\to X\times_{\A_P}\A_Q
\]
is strict \'etale.
We set $Y':=Y\times_{\A_P}\A_Q$.
There are canonical isomorphisms
\[
\Blow_{Z'} {X'}
\simeq
\Blow_{Y'}(Y'\times \A^r) \times_{Y'\times \A^r}X'
\simeq
\Blow_{\{0\}}(\A^r)\times_{\A^r}X'.
\]
Together with \eqref{blow.2.1}, we obtain the desired isomorphism.
\end{proof}

\begin{example}
\label{blow.4}
The conclusion of Lemma \ref{blow.1} is wrong if we do not assume $Z\in \lSm/S$.
For example, suppose
\[
X:=(\A^2,H_1+H_2),
\;
X':=(\Blow_{\{0\}}\A^2,\widetilde{H}_1+\widetilde{H}_2+E),
\text{ and }
Z:=X\times_{\A^2}\{0\},
\]
where $H_1$ and $H_2$ are the axes, $\widetilde{H}_1$ and $\widetilde{H_2}$ are their strict transforms, and $E$ is the exceptional divisor.
While $\Blow_Z X\times_X X'$ is not irreducible, $\Blow_{Z'}X'\simeq X'$ is irreducible.
\end{example}

\bibliography{bib}
\bibliographystyle{siam}

\end{document}